\documentclass[11pt]{amsart}
\usepackage{amsmath,amsthm,amsfonts,amssymb,mathrsfs,bm}
\usepackage{amsmath,amsthm,amsfonts,amssymb,bm,wasysym}
\usepackage{epsfig}
\usepackage[usenames]{color}
\usepackage{verbatim}
\usepackage{hyperref}
\usepackage{multicol}
\usepackage{graphicx}
\usepackage{comment}
\usepackage{float}
\usepackage{color}
\usepackage{enumerate}
\usepackage[normalem]{ulem}

\usepackage[color=green, textsize=tiny]{todonotes}

\usepackage{fullpage}

\newtheorem{theorem}{Theorem}[section]
\newtheorem{definition}[theorem]{Definition}

\newtheorem{lemma}[theorem]{Lemma}
\newtheorem{proposition}[theorem]{Proposition}
\newtheorem{corollary}[theorem]{Corollary}
\newtheorem{remark}[theorem]{Remark}

\numberwithin{equation}{section}

\newcommand{\1}{{\text{\Large $\mathfrak 1$}}}

\title{On the intersection of critical percolation clusters and other tree-like random graphs}

\author{Amine Asselah}
\address{Universit\'e Paris-Est, LAMA, UMR 8050, UPEC, UPEMLV, CNRS, F-94010 Cr\'eteil}
\email{asselah@u-pec.fr}

\author{Bruno Schapira}
\address{Aix-Marseille Universit\'e, CNRS, I2M, UMR 7373, 13453 Marseille, France}
\email{bruno.schapira@univ-amu.fr}

\begin{document}
\maketitle 

\begin{abstract}
We study intersection properties of two or more independent tree-like random graphs. Our setting encompasses critical, possibly long range, Bernoulli percolation clusters, incipient infinite clusters, as well as critical branching random walk ranges. We obtain sharp excess deviation bounds on the  number of intersection points of two or more clusters, under minimal assumption on the two-point function.  
The proofs are based on new bounds on the $n$-point function, in case of critical percolation, and on the joint moments of local times of branching random walks. 
\\
MSC 2020 \emph{subject classifications.} 60J80, 60K35. 
\end{abstract}

\section{Introduction}
Excess intersection of random walks is a classical theme in probability theory. In particular, large deviations bounds for the intersection of two independent simple random walk ranges in dimension five and higher can be found in 
Khanin, Mazel, Schlosman and Sina\"i~\cite{KMSS}. Later, van den Berg, Bolthausen and den Hollander~\cite{BBH}, have obtained a Large Deviation Principle for the intersection of two Wiener sausages in a finite time horizon, in any dimension. In dimension five and higher, the present authors were able in~\cite{AS23} to extend their estimates to an infinite time horizon in the discrete setup. In this paper we consider the more intricate setting of tree-like random graphs. Our main results are stretched exponential upper bounds for the intersection of critical percolation clusters, together with matching upper and lower bounds for the intersection of ranges of critical branching random walks. 

\vspace{0.2cm}
Our analysis partly relies on techniques that we developed for tackling similar questions for simple random walks  \cite{AS17,AS21,AS23,AS23b}, and on related results for branching random walks \cite{AS24+,ASS23, LGL,LGL2,Zhu,ZhuECP}. However, there are a number of important novelties here. First we adapt our arguments to the setting of critical percolation clusters, including incipient infinite clusters (IIC) in high dimension. Furthermore, we prove matching upper and lower bounds for the intersection of two Branching random walks in all dimensions, except for the critical dimension eight, where we get different exponents in the upper and lower bounds. Finally we manage to treat the intersection of more than two random sets. The latter relies crucially on some new upper bounds on the $n$-point function (in the setting of percolation clusters), see Proposition~\ref{prop.npoint}, and on the joint moments of the local times of Branching random walks, see Proposition~\ref{prop.localtimesBRW}. In the case of percolation clusters, this can be thought of as a kind of integrated version of the famous tree-graph inequality of Aizenman and Newman~\cite{AN84}, where we unravel a new labelled tree structure. Likewise,  in the case of Branching random walks, our approach complements the exact and more intricate  diagrammatic expansion derived by Angel, Hutchcroft and J\'arai~\cite{AHJ}. In both cases, integrating the occupation field against a test function and optimizing  brings into the game a capacity, whose distinguished properties play an important role in our analysis.

\subsection{Main results: case of two fractal clusters}
Our results may be formulated in a quite general setting, namely we just need some rough control on the volume growth of balls of the ambient space, some polynomial upper bound on the two-point function, and the BK inequality.

To be more specific, 
consider $(V,d_V)$ a countable metric space with polynomial volume growth, in the sense that if for $x\in V$ and $r>0$, 
$$B(x,r) = \{y\in V:d_V(x,y)\le r\},$$
there exist $d\ge 1$ (not necessarily integer), and positive constants $c_1,c_2$, such that
\begin{equation}\label{poly}
c_1\cdot r^d \le |B(x,r)|\le c_2\cdot r^d, \quad \textrm{for all }x\in V,\ r>0.  
\end{equation}
(See e.g.~\cite{Bar} for examples of graphs satisfying this hypothesis for each real $d\ge 1$.) The hypothesis~\eqref{poly} will be in force in the whole paper, and will therefore not be recalled anymore. A prominent example is of course when $V=\mathbb Z^d$, and the distance is given by the Euclidean norm, which is denoted by  $\|\cdot\|$.

Then we consider $\mathcal G$ a random graph with deterministic vertex set $V$ and possibly random edge set. Write $\{x\longleftrightarrow y\}$ for the event that two vertices $x,y\in V$ are connected in $\mathcal G$, and define the {\it two-point function} by 
$$\tau(x,y)= \mathbb P(x\longleftrightarrow y).$$ 
Define also for $\alpha>0$, the function $g_\alpha$, by
$$g_\alpha(x,y) = \frac 1{1+d_V(x,y)^\alpha},\quad x,y\in V.$$
We will require that for some $\alpha \in (d/2,d)$,
\begin{equation}\label{Halpha}
\begin{array}{cc}
\textrm{there exists } C>0, \textrm{ such that for all }x,y\in V, \\
\tau(x,y) \le C\cdot  g_\alpha(x,y). 
\end{array}
\tag{$\mathcal{H}_\alpha$}
\end{equation}
Furthermore, we always assume that $\mathcal G$ satisfies the van den Berg--Kesten (BK) inequality~\cite{BK}, whose definition is  recalled later, see Definition~\ref{def.BK}.

\vspace{0.2cm}
Fix arbitrarily a vertex $0\in V$, and denote by $\mathcal C_0$ its connected component, also sometimes called cluster. 
Our first result provides a stretched exponential moment bound for the intersection of two independent clusters.  

\begin{theorem}\label{LD.intersection.finite}
Assume that $\mathcal G$ satisfies the BK inequality and \eqref{Halpha}, for some $\alpha\in (\tfrac{3d}{4},d)$. Let $\mathcal C_0$ and $\widetilde{\mathcal C_0}$ be two independent copies of the connected component of $0$ in $\mathcal G$. 
There exists $\kappa>0$, such that  
$$\mathbb E\Big[\exp\big(\kappa\cdot|\mathcal C_0 \cap \widetilde {\mathcal C}_0|^{\frac{2\alpha}{d}-1}\big)\Big]<\infty.   $$  
\end{theorem}
The result is mainly designed to apply to critical Bernoulli percolation on $\mathbb Z^d$, for which the BK inequality always holds. Furthermore, it is believed that~\eqref{Halpha} should hold in any dimension (see e.g.~\cite[Ch.9]{Gri}), with an optimal $\alpha$ that should take its mean-field value $\alpha=d-2$, when $d>6$. 
Currently a proof of this latter fact is known in dimension $d>10$ for the usual nearest-neighbor lattice, and in any dimension $d>6$ provided the lattice is sufficiently spread out (i.e.~with edges connecting all pairs of vertices at distance at most $L$ one from each other, with $L$ sufficiently large), see~\cite{FvdH17,FvdH17b}, or~\cite{DCP} for a recent proof not using lace expansion,  and~\cite{BS85, Hara08, HvdHS,HS90,HS00,HvdHSa} for earlier results, as well as~\cite{HvdH} for a thorough account on high dimensional percolation. Note that in the mean-field regime where $\alpha=d-2$, the condition $\alpha>3d/4$, is equivalent to $d\ge 9$. This condition may be better understood, knowing that it is only in dimension nine and higher that the intersection of two independent IIC is almost surely finite, but we shall comment more on this condition in a moment.

\vspace{0.2cm}
Interestingly, we do not make any hypothesis on the degrees of the graph. In particular our setting covers the case of long-range percolation on $\mathbb Z^d$, for which the asymptotic behavior of the two-point function at criticality has been obtained   in~\cite{CS15,CS19,HvdHS,Hut21,Hut24+}, see also~\cite{Hut24} for similar results on the hierarchical lattice in the mean-field regime. Strikingly, in this model the range of possible values for $\alpha$ spans the whole admissible interval $(3d/4,d)$.

\vspace{0.2cm} Some other examples that might be particularly interesting to look at, and which have been extensively studied in recent years, are the closely related graphs formed by the excursion sets of a Gaussian Free Field, the (vacant) set of Random Interlacements, and the (vacant) set of a random walk loop soup. In the first of these examples an explicit expression of the two-point function has been computed by Lupu~\cite{Lupu16}, and its asymptotic behavior was identified in~\cite{AS19} for the vacant set of a random walk loop soup. 
Interestingly, for level sets of the Gaussian free field, the capacity of a  cluster is now well understood thanks to~\cite{DPR,DPR2}, contrarily to the case of Bernoulli percolation on $\mathbb Z^d$. Note that 
one difficulty in these models is that the BK inequality is not satisfied, which typically complicates the analysis, see however~\cite{CD23,CD24,Werner21}, and references therein, for recent major progresses, which circumvent this difficulty.

\vspace{0.2cm} Our results for branching random walks, see Theorem~\ref{thm.BRW} below, provide evidence that the exponent $\tfrac{2\alpha}{d}-1$ should be sharp, in the sense that if the two-point function really decays as $d_V(x,y)^{-\alpha +o(1)}$, then the stretched exponential moment with any larger exponent should be infinite.  However, in the setting of critical percolation, proving this essentially amounts to estimate the probability that a single cluster  
covers a positive fraction of a ball centered at the origin, which appears to be a very challenging problem.

\vspace{0.2cm}
As for the condition $\alpha\in (3d/4,d)$, first note that a direct computation shows that if $\tau(x,y)$ decays as $d_V(x,y)^{-\alpha}$, for some $\alpha>0$, then 
$$\mathbb E\big[|\mathcal C_0 \cap \widetilde{\mathcal C}_0|\big]<\infty \quad \Longleftrightarrow \quad \alpha>d/2. $$ 
In particular, when $\alpha\le d/2$, the tail distribution cannot be stretched exponential. In fact, based on our computation for BRWs, we conjecture that if $\alpha < 3d/4$, the tail distribution should decay at a polynomial speed. The precise exponent of the polynomial is not easy to guess in general, since in particular the analogy with Branching random walks might not be valid outside the mean-field regime. However, as we will see later, for Bernoulli bond percolation in a sufficiently spread out lattice in dimension seven, we are able to show that it decays as $t^{-4}$, thanks to the computation of the one-arm exponent by  Kozma and Nachmias~\cite{KN}, see Section~\ref{sec.lowdim}. 
Actually, it might also be possible to compute the exponent in case of Bernoulli site percolation on the triangular lattice, thanks to~\cite{LSW} and~\cite{SW}, but we have not tried to push further in this direction.   
Finally, in the critical case $\alpha=3d/4$ (which is equivalent to $d=8$ when $\alpha = d-2$), we can show that for some constant $\kappa>0$, 
$$\mathbb E\big[\exp(\kappa\cdot |\mathcal C_0\cap  \widetilde{\mathcal C}_0|^{1/3})\big]<\infty,$$
see the proof in Section~\ref{sec.dim8}. 
In particular the tail is at least stretched exponential, but we cannot identify the optimal exponent, even in the  case of BRWs.


\vspace{0.2cm} 
The proof of Theorem~\ref{LD.intersection.finite} is based on a multiscale analysis, which is adapted from the techniques developed in~\cite{AS23} and in previous works. In particular it relies primarily on an exponential moment bound for additive functionals of the occupation field of a cluster, which is the content of Theorem~\ref{thm.expmoment} below. This result may be seen as an extension of Kac's moment formula for random walks, see e.g.~\cite[Proposition 2.9]{Sz}, and is interesting on its own. 
Given a vertex $x\in V$, we denote by $\mathcal C(x)$ its connected component in $\mathcal G$: 
$$\mathcal C(x) = \{y\in V : x\longleftrightarrow y\}. $$ 
For $\varphi:V\to \mathbb R$, and $K:V\times V\to \mathbb R$, 
we write 
$K*\varphi(x) = \sum_{y\in V} K(x,y)\varphi(y)$. 
and let  $\|\varphi\|_\infty = \sup_{x\in V} |\varphi(x)|$. 

\begin{theorem}\label{thm.expmoment}
Assume that $\mathcal G$ satisfies the BK inequality and \eqref{Halpha}, for some $\alpha\in(d/2,d)$. There exists a constant $\kappa>0$, such that for any $\varphi:V\to [0,\infty)$, satisfying $\|g_{2\alpha- d}*\varphi\|_\infty \le 1$, and any $x\in V$,
$$\mathbb E\left[\exp\left(\kappa \cdot \sum_{y\in \mathcal C(x)} \varphi(y)\right)\right] \le 2. $$ 
\end{theorem}
We stress that the range of applications of this result is not merely limited to the  proof of Theorem~\ref{LD.intersection.finite}. In particular we shall see in Section~\ref{subsec.applications} below that it has some other interesting consequences concerning the geometry of a typical cluster. Its proof is a variant of the tree-graph inequality method of Aizenman and Newman~\cite{AN84}. However, here we use a slightly different induction argument to bound the moments of $\sum_{y\in \mathcal C(x)} \varphi(y)$, very much as in~\cite{ASS23} in the setting of branching random walks.

\vspace{0.2cm}
As a matter of fact,  Theorems~\ref{LD.intersection.finite} and~\ref{thm.expmoment} can be extended to the case of the Incipient Infinite Cluster (IIC). To be more precise, consider Bernoulli (bond) percolation on the lattice $\mathbb Z^d$, with either $d>10$ and nearest neighbor edges, or $d>6$ and edges between any pair of vertices at distance at most $L$ one from each other, with $L$ sufficiently large. In this setting, the IIC can be defined as the limit in distribution of the cluster of the origin, conditioned on being connected to $x$, as we let $x$ tend to infinity. The existence of this limit has been shown in~\cite{FvdH17,HvdHH,vdHJ}. Note that the IIC has also been defined in dimension two by Kesten~\cite{Kes} (see also \cite{Jar}), but  this case falls in another regime, which we shall not consider here.

For $\varphi:\mathbb Z^d\to [0,\infty)$, we let  $\|\varphi\|_1=\sum_{x\in \mathbb Z^d}\varphi(x)\in [0,\infty]$.

\begin{theorem}\label{thm.IIC}
Let $\mathcal C_\infty$ be the IIC of Bernoulli percolation on $\mathbb Z^d$, either with $d>6$ and sufficiently spread-out lattice,  or $d>10$ in the nearest neighbor model. There exist positive constants $C$ and $\kappa$, such that for any $\varphi : \mathbb Z^d \to [0,\infty)$,  satisfying  $\|g_{d-4} * \varphi\|_\infty \le 1$,
\begin{equation}\label{IIC.1}
\mathbb E\left[\exp\left(\kappa \cdot \sum_{x\in \mathcal C_\infty} \varphi(x)\right)\right] \le 1+C \, \|\varphi\|_1.
\end{equation}
Furthermore, if $d\ge 9$ and $\widetilde {\mathcal C}_\infty$ denotes an independent copy of $\mathcal C_\infty$, then
\begin{equation}\label{IIC.2} 
\mathbb E\left[\exp\left(\kappa \cdot|\mathcal C_\infty \cap \widetilde {\mathcal C}_\infty|^{1-\frac 4d}\right)\right] <\infty.  
\end{equation}
\end{theorem}
We note that in the sufficiently spread-out model with $d\in \{7,8\}$ the intersection of two independent IIC has infinite volume almost surely. Indeed, this can be shown using a second moment method and the estimates on the two-point function from~\cite[Theorem 1.3]{vdHJ}.

\subsection{Some applications} \label{subsec.applications}
We present now some other immediate applications of Theorem~\ref{thm.expmoment}, concerning the geometry of a typical cluster.

First recall that for critical Bernoulli percolation on a Euclidean lattice, it is widely believed that the tail distribution of $|\mathcal C_0|$ should decay as 
$$\mathbb P(|\mathcal C_0|\ge n) \approx n^{-1/\delta+o(1)},$$
where $\delta$ is supposed to be universal, 
see~\cite[Ch.~9]{Gri}. It has been shown by Aizenman and Barsky~\cite{AB} that if it exists, $\delta$ must be at least $2$, and it is also known that $\delta =2$ on high dimensional hypercubic lattices, see~\cite{FvdH17} and references therein, or more generally under the triangle condition, see~\cite{AN84,BA,Hut22} and~\cite{Hut24+} for an application to long-range percolation. 
The next result shows that if one considers now the size of the cluster in a localized region, then the tail distribution is stretched exponential. This fact was already known for critical clusters restricted to balls, by~\cite{AN84}, but we can now extend it to arbitrary finite sets.

Given a function $g:V\times V\to \mathbb R$, $y\in V$, and $A\subset V$, we define 
$g(y,A) = \sum_{z\in A} g(y,z)$.  
\begin{corollary}\label{cor.1}
Assume that $\mathcal G$ satisfies the BK inequality and~\eqref{Halpha}, for some $\alpha\in (d/2,d)$. 
There exists $c>0$, such that for any $x\in V$,  any finite $A\subset V$, and any $t>0$, 
$$\mathbb P(|\mathcal C(x) \cap A|> t)
 \le 2 \cdot \exp\Big(-c \cdot \frac{t}{\sup_{y\in V} g_{2\alpha -d}(y,A)}\Big)
\le 2\cdot 
\exp\Big(-c^2 \cdot \frac{t}{|A|^{2(1- \frac{\alpha}{d})}}\Big).  $$ \end{corollary}
\begin{proof}
Let us start with the first inequality. Let  $\varphi (y)= \frac{\mathbf 1\{y\in A\}}{\sup_{z\in V} g_{2\alpha -d}(z,A)}$. Note that it satisfies $\| g_{2\alpha -d}* \varphi\|_\infty\le 1$, and that
$$\mathbb P(|\mathcal C(x) \cap A|> t) \le 
\mathbb P\Big(\sum_{y\in \mathcal C(x)} \varphi(y) > \frac{t}{\sup_{z\in V} g_{2\alpha -d}(z,A)}\Big). $$
It then suffices to apply Theorem~\ref{thm.expmoment}, and Chebyshev's exponential inequality. The second inequality follows from a simple rearrangement and the hypothesis of polynomial volume growth.  
\end{proof}

Now we mention another application involving the notion of $\beta$-capacity, which plays a central role in the proof of Theorem~\ref{LD.intersection.finite}. For $\beta\in (0,d)$, the $\beta$-capacity of a finite set $A\subset V$, is defined by 
$$\textrm{Cap}_\beta(A) = \left( \inf_\nu \sum_{x,y\in A} g_\beta(x,y) \nu(x)\nu(y) \right)^{-1},$$ 
where the infimum runs over probability measures $\nu$ supported on $A$. It is well known (but for reader's convenience we provide a full proof in the Appendix) that this functional is equivalent to the following one:
$$\widetilde{\textrm{Cap}}_\beta(A) = \sup \left\{\sum_{x\in A} \varphi(x) : \varphi :V\to [0,\infty), \,   \|g_\beta*\varphi\|_\infty\le 1, \varphi \equiv 0 \text{ on } V\setminus A\right\}, $$ 
in the sense that there are 
positive constants $c$ and $c'$ such that for any finite $A\subset V$, 
\begin{equation}\label{equiv.cap}
c\cdot \widetilde{\textrm{Cap}}_\beta(A)\le  \textrm{Cap}_\beta(A) \le c'\cdot  \widetilde{\textrm{Cap}}_\beta(A).
\end{equation}

\begin{corollary}
Assume that $\mathcal G$ satisfies the BK inequality and~\eqref{Halpha}, for some $\alpha\in (d/2,d)$. There exists a constant $\kappa>0$, such that for any $x\in V$, and any finite $A\subset V$, 
$$\mathbb P(A\subseteq \mathcal C(x))\le 2\cdot \exp\left(-\kappa \cdot {\rm Cap}_{2\alpha-d}(A)\right).  $$ 
Furthermore, the same inequality holds for the IIC on $\mathbb Z^d$, instead of $\mathcal C(x)$, with $\alpha = d-2$, when either $d\ge 11$ for the nearest neighbor model, or $d>6$ in the sufficiently spread out model.  
\end{corollary}
\begin{proof}
Let $\varphi$ be a function realizing the maximum in the definition of $\widetilde{\textrm{Cap}}_{2\alpha -d}(A)$, and note that  
$$\mathbb P(A\subseteq \mathcal C(x)) \le \mathbb P\Big(\sum_{y\in \mathcal C(x)} \varphi(y) \ge \sum_{y\in A}\varphi(y)\Big). $$ 
Then the result follows from Chebyshev'exponential inequality,  Theorem~\ref{thm.expmoment}, and~\eqref{equiv.cap}. The proof for the IIC is exactly the same. 
\end{proof}

\subsection{Extension to an arbitrary finite number of clusters}
One natural question one can ask concerns the extension of our main result to the intersection of more than two clusters. It turns out that this is a quite delicate question. Our result is as follows. 
\begin{theorem}\label{theo.generalcluster}
Assume that $\mathcal G$ satisfies the BK inequality and \eqref{Halpha}, for some $\alpha\in (d/2,d)$. Let $(\mathcal C_0^i)_{i\ge 1}$ be independent copies of the connected component of $0$ in $\mathcal G$. 
\begin{itemize}
\item[$(i)$] If $\alpha \in (\frac{(k+1)d}{2k},\frac{kd}{2(k-1)})$, for some $k\ge 2$, then there exists $\kappa>0$, such that  
$$\mathbb E\left[\exp\left(\kappa \cdot |\mathcal C_0^1 \cap \dots\cap \mathcal C_0^k|^{\frac{2\alpha}{d}-1}\right)\right]<\infty.  $$
\item[$(ii)$] If $\alpha = \tfrac{kd}{2(k-1)}$, for some $k\ge 3$, then there exists $\kappa>0$, such that for any $t>e$, 
$$\mathbb P(|\mathcal C_0^1 \cap \dots\cap \mathcal C_0^k|>t)\le \exp\Big(-\kappa (\frac t{\log t})^{\frac {2\alpha}d-1}\Big).$$
\end{itemize}
\end{theorem}
Note that if $\alpha>3d/4$ (which corresponds to Part $(i)$ with $k=2$) then the result is exactly the same as  Theorem~\ref{LD.intersection.finite}. As for Part $(ii)$, we believe that the logarithm appearing in our upper bound should not be there. 

\vspace{0.1cm}
The main new ingredient in the proof is a general upper bound on the $n$-point function, which refines the tree-graph inequality of Aizenman and Newman~\cite{AN84}.

The tree-graph inequality expresses the probability that $n$ given points $x_1,\dots,x_n$ are all in the percolation cluster 
of another point $z$ as a double sum. The first sum is about a finite number of binary trees with $n$
leaves marked with $x_1,\dots,x_n$ and the root marked with $z$, whereas the second sum runs over the $n-2$ marks of the internal nodes different from the root, 
each mark running over $\mathbb Z^d$. Now, the object we sum is a product of Green's
functions evaluated at the differences of marks over edges of the tree. When dealing with the $n$-th power of the intersection of say $k$ independent clusters containing a fixed point $z$, we sum over $x_1,\dots,x_n$ the $k$-th power of the $n$-point function.
In order to transfer the $k$-th power over Green's functions, Jensen's inequality is powerful but cannot afford an infinite sum over internal nodes' marks. In our approach, we rule out this problem by expressing the $n$-point function as a single finite sum over a family of trees which we describe explicitly.


\vspace{0.2cm}
To be more precise now, we show that 
there exists a constant $C>0$, such that for any $x_1,\dots,x_n,z\in V$, one has with $\underline{x} = (x_1,\dots,x_n)$, 
\begin{equation}\label{npoint}
\mathbb P(x_1,\dots,x_n \in \mathcal C(z))\le  C^n \cdot \sum_{\mathfrak t \in \mathbb T_n} G_\alpha(\mathfrak t,\underline x) \cdot g_\alpha(x_{\mathfrak t},z), 
\end{equation}
where $\mathbb T_n$ is a certain set of finite plane trees with $n$ marked and labelled vertices; $x_{\mathfrak t}\in\{x_1,\dots,x_n\}$ is the point associated to the first labelled  vertex of $\mathfrak t$ in the lexicographical (or depth-first search) order, and $G_\alpha(\cdot,\cdot)$ is some function, which is equal to the product over the edges $\{i,j\}$ of an auxiliary tree associated to $\mathfrak t$ with vertex set $\{1,\dots,n\}$, of terms of the form $g_{2\alpha - d}(x_i,x_j)$; see Section~\ref{sec.manyclusters} and Proposition~\ref{prop.npoint} for more precise definition and statement. 

\vspace{0.2cm}
We stress that the trees involved in $\mathbb T_n$ are not necessarily binary, in particular our proof follows a slightly different approach from the one in~\cite{AN84}, which in particular has the advantage of being also well suited to the setting of Branching random walks, as we shall see later.

\vspace{0.2cm}
Finally let us mention that a similar result as Theorem~\ref{theo.generalcluster} could be proved as well for the intersection of independent IIC. Since the argument is entirely analogous to the proof of Theorem~\ref{thm.IIC}, we refrain from adding more details here.

\subsection{Intersection of clusters of different types}\label{subsec.twotypes}
Our techniques allow to consider as well the intersection of two  clusters of different types. For instance, we can show the following result. 
\begin{theorem}
Let $\mathcal G$ and $\widetilde{\mathcal G}$ two random graphs on the same metric space $(V,d_V)$, satisfying both the BK inequality, and respectively~\eqref{Halpha} and $(\mathcal H_\beta)$, for some $\alpha,\beta \in(\frac{d}{2},d)$, with $\alpha\le \beta$, and $2\alpha + 2\beta > 3d$. Let $\mathcal C_0$ and $\widetilde{\mathcal C}_0$ the clusters of the origin in $\mathcal G$ and $\widetilde{\mathcal G}$ respectively. There exsits a constant $\kappa>0$, such that 
\begin{equation}\label{intersection.twodifferent}
\mathbb E\Big[\exp\big(\kappa\cdot  |\mathcal C_0\cap \widetilde{\mathcal C}_0|^\gamma\big)\Big]<\infty, \quad \textrm{with}\  
\gamma = \frac{2\alpha - d}{d+ 2\alpha - 2\beta}.
\end{equation}
\end{theorem}
The proof uses the same arguments as for Theorem~\ref{LD.intersection.finite} $(i)$. Some more details will be given in Remark~\ref{rem.alphabeta}. A heuristic explanation of the exponent $\gamma$ is as follows. If one assumes that the upper bound given by Corollary~\ref{cor.1} is sharp, then for each cluster, the probability to cover at least $t$ points of a ball $B(0,r)$ would decay as $\exp(-t / r^{2d-2\alpha})$, when it satisfies \eqref{Halpha}. Therefore, if the cluster satisfying~\eqref{Halpha} spends a time $t_1$ there, and the cluster satisfying $(\mathcal H_\beta)$ a time $t_2$, the two costs are balanced if $t_1\, r^{2\alpha} = t_2 \, r^{2\beta}$. Furthermore, using independence between the two clusters, and making the usual approximation that the points of each cluster are uniformly spread in the ball, the overlap between the two clusters would then be of order $t_1t_2/r^d$, and thus we want $t_1t_2/r^d = t$. Given this, the cost is minimized if $t_1/r^d$ is maximized, i.e. of order $1$, which yields a  cost $\exp(-t^\gamma)$, with $\gamma$ as above.

\vspace{0.1cm}
Note that this heuristic suggests that if one considers more than two clusters, say $k\ge 3$ clusters satisfying \eqref{Halpha} with parameters respectively $\alpha_1\ge \dots\ge \alpha_k$, then in case when $2(\alpha_1+\alpha_2)>3d$, the cost should be of the same order as if we would just ask for the two first clusters (associated to $\alpha_1$ and $\alpha_2$) to share more than $t$ common points, because once the second cluster realizes a density of order $1$ in a ball $B(0,r)$, the cost for the other clusters to also cover a fraction of order $1$ of this ball is smaller, and this just makes the total intersection of all clusters decay by a constant factor. On the other hand when $2(\alpha_1+\alpha_2)\le 3d$, we expect a different scenario. Assuming that for some $3\le i\le k$ (necessarily unique), one has $ 2(\alpha_1+\dots+\alpha_i) > (i+1)d$ 
and $2(\alpha_1+ \dots +\alpha_{i-1})\le id$, we expect a cost of order $\exp(-t^{\gamma_i})$, with 
$\gamma_i = \tfrac{2\alpha_i -d}{d+2(i-1)\alpha_i - 2(\alpha_1+\dots+\alpha_{i-1})}$.
However, except for some cases (e.g. when all $\alpha_i$'s are equal), proving this result in full generality seems to require new ideas.

\subsection{Intersection of Simple Random Walk (SRW) ranges} 
As a warmup before we investigate the more difficult case of BRW, let us consider the intersection of SRW ranges. Recall the famous result of Erd\"os and Taylor~\cite{ET}, which asserts that if $(R_\infty^i)_{i\ge 1}$, are independent SRW ranges, then 
\begin{align*}
 |R_\infty^1 \cap R_\infty^2| = \infty \ \textrm{ a.s.} \quad & \Longleftrightarrow \quad d\le 4. \\
 |R_\infty^1 \cap R_\infty^2\cap R_\infty^3| = \infty \ \textrm{ a.s.} \quad & \Longleftrightarrow\quad  d\le 3. \\
|R_\infty^1 \cap R_\infty^2\cap R_\infty^3 \cap R_\infty^4| = \infty\ \textrm{ a.s.} \quad & \Longleftrightarrow \quad d\le 2.
\end{align*}
While these results can be proved using an elementary second moment method, it is a much more difficult task to estimate the tail distribution of the intersection of these ranges (in case they are a.s. finite). Here we provide the following answer to this question. 

\begin{theorem}\label{theo.SRW}
Let $R_\infty^k$, $k\ge 1$, be independent SRW ranges on $\mathbb Z^d$. 
\begin{enumerate}  
\item[(i)] If $d\ge 5$, then for any $k\ge 2$, there exist positive constants $c_1,c_2$, such that for any $t>1$, 
$$\exp(-c_1\, t^{1-\frac 2d}) \le \mathbb P(|R_\infty^1\cap \dots \cap R_\infty^k|>t) \le  \exp(-c_2\, t^{1-\frac 2d}). $$ 
\item[(ii)] If $d=4$, then for any $k\ge 3$, there exist positive constants $c_1,c_2$, such that for any $t>e$, 
$$\exp(-c_1\, \sqrt{t} ) \le \mathbb P(|R_\infty^1\cap \dots \cap R_\infty^k|>t) \le  \exp\big(-c_2\, (\frac{t}{\log t})^{1/2}\big).$$
\item[(iii)] If $d=3$, then for any $k\ge 4$, there exist positive constants $c_1,c_2$, such that for any $t>e$, 
$$\exp(-c_1\, t^{1/3}) \le \mathbb P(|R_\infty^1\cap \dots \cap R_\infty^k|>t) \le  \exp(-c_2\, (\frac t{\log t})^{1/3}). $$
\end{enumerate}
\end{theorem}
We note that prior to this result and until recently, not much was known, apart from the lower bounds which are somewhat standard estimates. Concerning the upper bounds, a weaker version of Part $(i)$ was first proved in~\cite{KMSS}, with the exponent $1-2/d$ replaced by $1-2/d-\varepsilon$, for arbitrarily small $\varepsilon>0$ (with the constant $c_2$ depending on $\varepsilon$). It has remained an open problem for about thirty years to remove the factor  $\varepsilon$ from this estimate and identify the right constant in the exponential, until it could be solved in~\cite{AS23,BBH}.

\subsection{Intersection of BRW ranges} 
We now discuss the more delicate case of critical Branching random walks on $\mathbb Z^d$, starting with some definition. 

\vspace{0.1cm}
Fix a probability measure $\mu$ with mean one, such that $\mu(1)<1$, and  $\sum_{k\ge 0} e^{ck}\mu(k)<\infty$, for some $c>0$. Denote by $\mathcal T_c$ the Bienaym\'e-Galton-Watson (BGW) tree with offspring distribution $\mu$, and consider the associated BRW, which we view here as the random walk $(S_v)_{v\in \mathcal T_c}$, indexed by $\mathcal T_c$. For simplicity we will assume that the jumps of the walks are distributed according to the uniform measure on the neighbors of the origin, but straightforward adaptations of our proofs would work as well for centred and finitely supported measures. We also denote by $\mathcal T_\infty$ either the infinite invariant BGW tree, see e.g.~\cite{ASS23} for a definition, or alternatively the BGW tree conditioned to be infinite, also called Kesten's tree, see~\cite{Kestree}. These two trees are made of a spine, which is a copy of $\mathbb N$, to which are attached independent critical BGW trees at each vertex of the spine (more precisely the offspring distribution of vertices on the spine differs from the one of other vertices, but this is irrelevant in all our proofs). 
Then we denote by $\mathcal R_c$ and $\mathcal R_\infty$ the ranges of respectively a critical BRW and a random walk indexed by $\mathcal T_\infty$, all starting from the origin.

\vspace{0.2cm}
Now consider $(\mathcal R_\infty^i)_{i\ge 1}$ a sequence of independent random variables distributed as $\mathcal R_\infty$, on $\mathbb Z^d$, with $d\ge 5$. Similarly as for the SRW, a second moment method yields  
\begin{align*}
|\mathcal R_\infty^1 \cap \dots \cap  \mathcal R_\infty^k| < \infty\ \textrm{ a.s.} \quad & \Longleftrightarrow \quad k> \frac{d}{d-4}.  
\end{align*}
Our main result reads as follows. 

\begin{theorem}\label{thm.BRW}
Let $(\mathcal R_c^i)_{i\ge 1}$ and $(\mathcal R_\infty^i)_{i\ge 1}$, be independent copies of $\mathcal R_c$ and $\mathcal R_\infty$ respectively, on $\mathbb Z^d$, with $d\ge 5$. 
\begin{enumerate}
    \item[(i)] If $d=7$ or $d\ge 9$, then for any $k> \frac d{d-4}$, there exist positive constants $c_1,c_2$, such that for any $t>1$,  
    $$\exp(-c_1\cdot t^{1- \frac 4d})\le \mathbb P(|\mathcal R_c^1 \cap\dots\cap \mathcal  R_c^k|>t) \le \mathbb P(|\mathcal R_\infty^1 \cap\dots\cap \mathcal R_\infty^k|>t) \le  \exp(-c_2\cdot t^{1-\frac 4d}). $$
    \item[(ii)] If $d\in \{5,6,8\}$, then for any $k> \frac d{d-4}$, there exist positive constants $c_1,c_2$, such that for any $t>e$,
    $$\exp(-c_1\cdot t^{1- \frac 4d})\le \mathbb P(|\mathcal R_c^1 \cap\dots\cap \mathcal  R_c^k|>t) \le \mathbb P(|\mathcal R_\infty^1 \cap\dots\cap \mathcal R_\infty^k|>t) \le  \exp\Big(-c_2\cdot (\frac{t}{\log t})^{1-\frac 4d}\Big). $$
\end{enumerate}
\end{theorem}
We note that in dimension $d\ge 9$ the upper bound for general $k\ge 2$ follows from the upper bound for $k=2$, which can be done by following a similar argument as in~\cite{AS23}, using the exponential moment bound already proved in~\cite{ASS23}. In lower dimension we follow the same strategy as for the proofs of Theorems~\ref{theo.generalcluster}  and~\ref{theo.SRW}. In particular this requires us  to prove an upper bound on the joint moments of the local times, which is similar to~\eqref{npoint}, and might also be more tractable in practice than the exact diagrammatic expansion proved in~\cite{AHJ}. 
More precisely, it takes the following form: for $x\in \mathbb Z^d$, denote by 
\begin{equation} \label{localtime.def}
\ell_c(x) = \sum_{v\in \mathcal T_c} \1\{S_v = x\}, \quad{and}\quad \ell_\infty(x) = \sum_{v\in \mathcal T_\infty} \1\{S_v = x\},
\end{equation}
the local times at $x$ respectively for a critical BRW and a random walk indexed by $\mathcal T_\infty$. 
Then there exists a constant $C>0$, such that for any $x_1,\dots,x_n,z\in \mathbb Z^d$, (possibly with repetition), one has with $\underline x=(x_1,\dots,x_n)$,
$$\mathbb E_z\Big[\prod_{i=1}^n \ell_c(x_i)\Big] 
\le C^n \sum_{\mathfrak t \in \mathbb T_n} G_{d-2}(\mathfrak t,\underline x) \cdot g_{d-2}(x_{\mathfrak t},z), $$
where $\mathbb T_n$, $x_\mathfrak t$, and $G_{d-2}$ are defined exactly as in the setting of percolation clusters, and 
$$\mathbb E_z\Big[\prod_{i=1}^n \ell_\infty(x_i)\Big] 
\le C^n \sum_{\mathfrak t \in \mathbb T_n} \widetilde G(\mathfrak t,\underline x,z),$$
with $\widetilde G$ some appropriate  modification of the function $G_{d-2}$,
see Proposition~\ref{prop.localtimesBRW} and the definitions preceding it for a more precise statement.

\vspace{0.2cm}
Concerning the lower bounds, we show that the required intersection can be done, at the right cost, in a ball with radius of order $t^{1/d}$, centered at the origin, on which all the BRWs spend a time of order the volume of the ball. The idea for proving this is to use the notion of waves, which were introduced in~\cite{AS24+}. Roughly speaking  waves for a branching random walk play the role of excursions for a simple random walk. We show that during each wave, between $B(0,t^{1/d})$ and $B(0,2t^{1/d})^c$, the typical number of new sites visited is of order $t^{4/d}$. Thus one needs order $t^{1-4/d}$ waves, and we show that the cost for this is exponentially small in the number of waves.

\vspace{0.2cm}
To conclude, note that the question of the tail distribution of the intersection of two or more critical ranges makes sense in any dimension $d\ge 1$. In this direction one can show the following.  
\begin{theorem}\label{BRW.lowdim}
One has 
\begin{equation*}
\mathbb P(|\mathcal R_c \cap \widetilde {\mathcal R}_c|>t) \asymp     
\left\{
\begin{array}{ll}
t^{-4/d} & \text{if }d\in \{1,2,3\} \\
t^{-1}\cdot (\log t)^{-2} & \text{if }d=4 \\
t^{-\frac 4{8-d}} & \text{if }d\in \{5,6,7\}. 
\end{array}
\right. 
\end{equation*}
Furthermore, if $d=8$, there exist positive constants $c_1$ and $c_2$, such that for any $t>1$, 
\begin{equation}\label{bounds.d8}
\exp(-c_1\sqrt t) \le \mathbb P(|\mathcal R_c \cap \widetilde {\mathcal R}_c|>t) 
\le \exp\big(-c_2\cdot t^{1/3}\big). \end{equation}
\end{theorem}
Here $f \asymp g$ means that $f/g$ is bounded from above and below by positive constants. Similar bounds could be shown for the intersection of more than two clusters. 
In dimensions $d\le 7$ the result follows from a soft second moment argument and known bounds on hitting probabilities. In dimension eight, which is critical for this model (see also~\cite{Baran} for the related question of non-intersection of two Branching random walks, where dimension eight is as well critical), we provide two proofs for the upper bound, a first one based on our new bounds on the moments of local times, and another more sophisticated one using similar ideas as for the proof of Theorems~\ref{LD.intersection.finite}  and~\ref{thm.BRW}. Interestingly the two arguments lead to the same exponent $1/3$. Nevertheless, it is still unclear for us, whether this gives the right order of decay. 

\vspace{0.2cm}
We note that one could also show using the same proof as for BRWs that on $\mathbb Z^8$, if the lattice is sufficiently spread-out, the same upper bound as in~\eqref{bounds.d8} holds for the tail distribution of the intersection of two independent critical percolation clusters. Likewise, thanks to the computation of the one-arm exponent in~\cite{KN}, one could   show that on $\mathbb Z^7$, the tail distribution decays as  $t^{-4}$ (here both the upper and lower bounds can be proved).   

\subsection{Intersection of mixtures of SRW and BRW ranges}
Similarly as for~\eqref{intersection.twodifferent}, we can also treat the intersection of one SRW and one critical BRW ranges in any dimension $d>6$. Moreover, by combining this with our results from the last two sections we deduce the following.

\begin{theorem}\label{theo.SRWBRW}
Let $(R_\infty^k)_{k\ge 1}$, and $(\mathcal R_c^k)_{k\ge 1}$, be respectively SRW and critical BRW ranges on $\mathbb Z^d$. 
\begin{enumerate}  
\item[(i)] If $d\ge 5$, then for any $i\ge 2$, and any $j\ge 1$, there exist positive constants $c_1,c_2$, such that for any $t>1$, 
$$\exp(-c_1\, t^{1-\frac 2d}) \le \mathbb P(|R_\infty^1\cap \dots \cap R_\infty^i \cap \mathcal R_c^1 \cap \dots\cap \mathcal R_c^j|>t) \le  \exp(-c_2\, t^{1-\frac 2d}). $$ 
\item[(ii)] If $d\ge 7$, then for any $k\ge 1$, there exist positive constants $c_1,c_2$, such that for any $t>1$, 
$$\exp(-c_1\, t^{1-\frac 2{d-2}} ) \le \mathbb P(|R_\infty^1\cap \mathcal R_c^1 \cap \dots \cap \mathcal R_c^k|>t) \le  \exp (-c_2\, t^{1-\frac 2{d-2}}). $$ 
\end{enumerate}
\end{theorem}
The upper bound in Part (i) is a direct  consequence of Theorem~\ref{theo.SRW}, while for the upper bound in (ii), it suffices to treat the case of one SRW and one BRW ranges, for which the proof is entirely similar to the one for~\eqref{intersection.twodifferent}. Likewise, the lower bounds can be obtained similarly as those in Theorems~\ref{theo.SRW} and~\ref{thm.BRW}.  

\vspace{0.1cm}
The only missing case, which we cannot treat with our present techniques, is the intersection of one SRW and two (or more) critical BRWs in dimension $6$. Note that the dimension five falls in another regime, in particular in this case $|R_\infty^1 \cap \mathcal R_\infty^1 \cap \mathcal R_\infty^2|$ is almost surely infinite. 

\vspace{0.1cm}
On the other hand, our techniques provide additional information on the 
occupation density of the walks in the region where the intersection takes place. To illustrate the type of results that one can show, let us focus here only on the case where one SRW intersects one BRW. For $t>0$, $r\ge 1$ and $\rho>0$, let 
\begin{equation}\label{Rtrrho}
\mathcal R_t(r,\rho) = \big\{z\in \mathcal R_\infty : |B(z,r)\cap \mathcal R_\infty|>\rho r^d\big\} \cap B\big(0,\exp(t^{\frac{d-4}{d-2}})\big). 
\end{equation}
\begin{theorem}\label{theo.scenario}
Assume $d\ge 7$. There exists $\beta>1$, such that the two following claims hold. 
\begin{itemize}
    \item[(i)] There exist  positive constants $\rho$ and $a<b$, such that 
    $$ \lim_{t\to \infty} 
    \mathbb P\Big(at^{\frac{d}{d-2}} \le |\mathcal R_{\beta t}(\beta t^{\frac 1{d-2}},\rho)|\le bt^{\frac{d}{d-2}} \ \big|\ |R_\infty \cap \mathcal R_\infty|>t\Big) =1. $$
    \item[(ii)] For any $\varepsilon>0$, there exists $\rho>0$, such that 
    $$\lim_{t\to \infty}
 \mathbb P\Big(|R_\infty \cap \mathcal R_{\beta t}(\beta t^{\frac 1{d-2}},\rho)|\ge (1-\varepsilon)t  \ \big|\ |R_\infty \cap \mathcal R_\infty|>t\Big) =1. $$    
\end{itemize}
\end{theorem}
In short, the result says that as $t\to \infty$, conditionally on having an intersection larger than $t$, a fraction arbitrarily close to one of the intersection takes place in a region of volume of order $t^{d/(d-2)}$, where the BRW realizes an occupation density of order one, at  scale $t^{1/(d-2)}$. This supports the idea that the trace of the BRW should look like a kind of Swiss cheese -- a picture which emerges for instance as two SRW intersect more than usual  (see~\cite{BBH}) -- filling a positive fraction of a ball of radius $t^{1/(d-2)}$, where on the other hand the SRW only realizes an occupation density of order   $t^{-2/(d-2)}$.

\subsection*{Plan of the paper} In Section~\ref{sec.BK} we recall the definition of the BK inequality, and prove some basic results involving the functions $g_\alpha$. In Section~\ref{sec.cap}, we state some elementary facts about $\beta$-capacities. The proofs are given partly in this section, partly in the appendix. Then in Sections~\ref{sec.proofmoment} and~\ref{sec.proof.main} we prove our two main results concerning the intersection of two independent random clusters, namely  Theorems~\ref{thm.expmoment} and~\ref{LD.intersection.finite} respectively.
In Section~\ref{sec.IIC} we prove Theorem~\ref{thm.IIC} about the extension to the incipient infinite cluster. Section~\ref{sec.manyclusters} is dedicated to the proof of Theorem~\ref{theo.generalcluster}  concerning the extension to an arbitrary finite number of clusters. Finally in  Section~\ref{sec.BRW} we prove Theorems~\ref{theo.SRW},~\ref{thm.BRW} and~\ref{theo.SRWBRW} about Simple random walks and Branching random walks, Section~\ref{sec.lowdim} deals with the proof of Theorem~\ref{BRW.lowdim} about intersections in low dimension, and in Section~\ref{sec.scenario} we prove the remaining Theorem~\ref{theo.scenario} describing the typical scenario under the intersection event.

\subsection*{Acknowledgments.}
We thank Perla Sousi for enlightening discussions at an early stage of this project, and for pointing at reference~\cite{Khosh}. We also thank Alexis Prevost for mentioning~\cite{CD23,CD24,Werner21} to us. The authors acknowledge support from the grant ANR-22-CE40-0012 (project Local).

\section{Preliminaries}
\subsection{The BK inequality} 
\label{sec.BK}
Here we recall a definition of the BK inequality, which is a standard tool in independent percolation, but which can be defined more generally for a random graph. First an event $E$ is said to be increasing if whenever $\{\mathcal C = A\}\subseteq E$, with $A$ some fixed graph (with vertex set $V$), then it also holds $\{\mathcal C =B\}\subseteq E$, for each $B$ containing $A$ as a subgraph. We say that a finite graph $A$ witnesses $E$, if $\{A\subseteq \mathcal C\}\subseteq E$. Given two increasing events $E$ and $F$, we say that they hold disjointly, and denote the corresponding event as $E\circ F$, if there exists two finite subgraphs of $\mathcal C$ with disjoint edge sets, that witness these two events. Note that this is also an increasing event, so that one can define inductively  $E_1\circ E_2\circ\dots \circ E_n$, for increasing events $E_1,\dots,E_n$. 
\begin{definition}[BK inequality] \label{def.BK} The random graph $\mathcal G$ is said to satisfy the BK inequality, if for any increasing events $E$ and $F$, one has 
$$\mathbb P(E\circ F) \le \mathbb P(E ) \cdot \mathbb P(F). $$ 
\end{definition}

\subsection{Some Green's function computation}
We prove here some basic results involving Green's functions $g_\alpha$ that we shall need later. 
\begin{lemma}\label{lem.npoint}
Assume $\alpha \in (d/2,d)$. There exists a constant $C>0$, such that for any $L\ge 2$, any $x_1,\dots,x_L,z\in V$, one has 
$$\sum_{y\in V} g_\alpha(z,y) \prod_{\ell =1}^L g_\alpha(y,x_\ell) \le C^L\cdot \sum_{\sigma \in \mathfrak S_L}
g_\alpha(x_{\sigma(1)},z) \prod_{\ell = 1}^{L-1} g_{2\alpha-d}(x_{\sigma(\ell)},x_{\sigma(\ell +1)}). $$ 
\end{lemma}
\begin{proof}
We prove the result by induction on $L$. For $L=2$, the result is well known, and follows from elementary computation. Assume now that it has been proved for some $L$, and let us prove it for $L+1$. Let $x_1,\dots,x_{L+1},z\in V$ be given, and define 
$$r = \min_{j=1,\dots,n} d_V(x_j,z).$$
We first bound the sum over $y\notin B(z,r/2)$. Let $i$ be such that $d_V(x_i,z)=r$. We shall use the induction hypothesis, together with  the fact that for  $y\notin B(z,r/2)$, one has $g_\alpha(z,y) \le 2^\alpha\,  g_\alpha(z,x_i)$. This yields, with $\mathfrak S_L^i$, the set of bijections from $\{2,\dots,L+1\}$ to $\{1,\dots,L+1\}\setminus \{i\}$, 
\begin{align}\label{sumsigma1}
\nonumber \sum_{y\notin B(z,r/2)} g_\alpha(z,y) \prod_{\ell =1}^L g_\alpha(y,x_\ell) & \le 2^\alpha \, g_\alpha(z,x_i) \sum_{y\in V} \prod_{\ell = 1}^L g_\alpha(y,x_\ell) \\
\nonumber & \le 2^\alpha C^L\cdot  g_\alpha(z,x_i) \sum_{\sigma\in \mathfrak S_L^i} 
g_\alpha(x_{\sigma(2)},x_i) \prod_{\ell = 2}^{L} g_{2\alpha-d}(x_{\sigma(\ell)},x_{\sigma(\ell +1)}) \\
& \le C^{L+1}  \sum_{\sigma\in \mathfrak S_L^i} g_\alpha(x_{\sigma_i(1)},z) \prod_{\ell = 1}^{L} g_{2\alpha-d}(x_{\sigma_i(\ell)},x_{\sigma_i(\ell +1)}),
\end{align}
denoting for any $\sigma \in \mathfrak S_L^i$, by $\sigma_i$ the permutation of $\mathfrak S_{L+1}$ sending $1$ to $i$, and coinciding with $\sigma$ on $\{2,\dots,L+1\}$, and choosing a constant $C>2^\alpha$.

Now it remains to bound the sum over $y\in B(z,r/2)$. Let $\sigma \in \mathfrak S_{L+1}$, be such that  
$$d_V(z,x_{\sigma(1)}) \le d_V(z,x_{\sigma(2)})\le  \dots \le  d_V(z,x_{\sigma(L+1)}). $$
It is immediate to see (e.g. by induction on $L$) that one has 
\begin{align}\label{sumsigma2}
\nonumber \sum_{y\in B(z,r/2)}  g_\alpha(z,y) \prod_{\ell =1}^L g_\alpha(y,x_\ell)
& \le C^{L+1}\cdot g_{2\alpha -d}(x_{\sigma(1)},z) \prod_{\ell = 1}^L g_\alpha(x_{\sigma(\ell)},x_{\sigma(\ell+1)}) \\
&\le C^{L+1}\cdot g_\alpha (x_{\sigma(1)},z) \prod_{\ell = 1}^L g_{2\alpha-d}(x_{\sigma(\ell)},x_{\sigma(\ell+1)}). 
\end{align}
Then combining~\eqref{sumsigma1} and~\eqref{sumsigma2} concludes the proof of the lemma.  
\end{proof}

For the next two results we specify the setting to the case where the ambient space is $\mathbb Z^d$, since we shall only use it in this case.

\begin{lemma}\label{lem.Green2}
Assume $d\ge 5$. There exists a constant $C>0$, such that for any $z,x_1,x_2\in \mathbb Z^d$, 
$$\sum_{y\in \mathbb Z^d } g_{d-2}(z,y)\,  g_{d-2}(x_1,y)\, g_{d-4}(x_2,y) \le C\cdot  g_{d-4}(x_1,x_2)\cdot \big(g_{d-4}(x_1,z) + g_{d-4}(x_2,z)\big). $$
\end{lemma}
\begin{proof}
Using translation invariance, one may always assume that $x_2= 0$. Now assume first that $\|x_1\| \le \|z\|$. Then note that 
\begin{align*}
&  \sum_{y\in B(0,\|x_1\|/2)^c} g_{d-2}(z,y)  \,  g_{d-2}(x_1,y)\, g_{d-4}(0,y) \\ 
& \le C g_{d-4}(0,x_1) \sum_{y\in \mathbb Z^d}g_{d-2}(z,y)\,  g_{d-2}(x_1,y) \le Cg_{d-4}(0,x_1)g_{d-4}(z,x_1). 
\end{align*}
On the other hand,  
\begin{align*}
& \sum_{y\in B(0,\|x_1\|/2)}g_{d-2}(z,y)  \,  g_{d-2}(x_1,y)\, g_{d-4}(0,y) \\ 
& \le Cg_{d-2}(z,0) g_{d-2}(x_1,0) \cdot \|x_1\|^4 \le Cg_{d-4}(x_1,z)g_{d-4}(0,x_1), 
\end{align*}
which proves the result in the case $\|x_1\|\le \|z\|$. A similar argument may be used in the other case $\|z\|\le \|x_1\|$, concluding the proof of the lemma. 
\end{proof}
Our next result is a direct consequence of this lemma. To state it we need some additional notation. Define inductively $T_n$ a set of rooted trees with $n$ vertices labelled in $\{1,\dots,n\}$, $n\ge 1$, as follows. Let $T_1$ be the unique rooted tree with one vertex labelled $1$. Then, given any $\mathfrak t\in T_n$, we associate two trees in $T_{n+1}$ by first adding an edge between a new vertex with label $n+1 $ to the root of $\mathfrak t$, and either rerooting the tree at this new vertex, or keeping the same root as in $\mathfrak t$. In particular $T_n$ has $2^{n-1}$ elements. Then define the function $f_n$ on $T_n\times (\mathbb Z^d)^n$, by 
$$f_n(\mathfrak t,x_1,\dots,x_n) = \prod_{(i,j)\in \mathcal E(\mathfrak t)} g_{d-4}(x_i,x_j),$$
where $\mathcal E(\mathfrak t)$ denotes the edge set of $\mathfrak t$. 
\begin{proposition}\label{prop.Green}
Assume $d\ge 5$. There exists a constant $C>0$, such that for any $L\ge 1$, any $z,x_1,\dots,x_L\in \mathbb Z^d$, 
$$\sum_{y_1,\dots,y_L\in \mathbb Z^d} 
g_{d-2}(z,y_1)\cdot\big(\prod_{\ell = 1}^{L-1}g_{d-2}(y_\ell,y_{\ell +1})\big)\cdot\big(\prod_{\ell = 1}^L
g_{d-2}(x_\ell,y_\ell)\big) \le C^L \sum_{\mathfrak t\in T_{L+1}} f_{L+1}(\mathfrak t,x_L,\dots,x_1,z). $$ 
\end{proposition}
\begin{proof}
We first observe that for any $y_{L-1}\in \mathbb Z^d$,
$$\sum_{y_L\in \mathbb Z^d}g_{d-2}(y_{L-1},y_L) g_{d-2}(x_L,y_L)\le Cg_{d-4}(y_{L-1},x_L),$$
which already proves the result when $L=1$. 
The general case $L\ge 1$ follows by induction, using Lemma~\ref{lem.Green2}. 
\end{proof}
\section{Basic facts about discrete Bessel-Riesz capacities}\label{sec.cap}
In this section we present basic facts about $\beta$-capacities, also called Bessel-Riesz capacities  in~\cite{Khosh}. We do not claim originality of the results, but since some of them might be difficult to find in the literature, we provide either precise references or in most cases full proofs, which we either include in this section when they are short, or defer to the Appendix for longer ones.

We fix in the whole section $\beta\in (0,d)$.  
We start with the following important result.  
\begin{lemma}\label{cap.translation}
For any finite sets $A,B \subset V$, one has the subadditivity property, 
$${\rm Cap}_\beta(A\cup B) \le {\rm Cap}_\beta(A) + {\rm Cap}_\beta(B). $$ 
Furthermore ${\rm Cap}_\beta$ is monotone for the inclusion of sets.
\end{lemma}
\begin{proof} The monotonicity follows  immediately from the definition, and a proof of the subadditivity property can be found in~\cite{AOSS}.  
\end{proof}

The next result provides the order of growth of the capacity of balls.  
\begin{lemma}\label{cap.ball}
There exist positive constants $c$ and $c'$, such that for any $x\in V$ and $r\ge 1$, 
$$c\cdot r^\beta \le {\rm Cap}_\beta(B(x,r))\le c'\cdot r^\beta.$$ 
\end{lemma}
\begin{proof}
For the upper bound, it suffices to notice that for any $y,z\in B(x,r)$, one has $g_\beta(y,z)\ge 1/(1+(2r)^{\beta})$. Then we deduce immediately from the definition that 
$\textrm{Cap}_\beta(B(x,r)) \le 1+(2r)^\beta$. 
As for the lower bound, let $\varphi(y) = \frac{\mathbf 1\{y\in B(x,r)\}}{\sup_z g_\beta(z,B(x,r))}$. 
By definition one has 
$$\widetilde{{\rm Cap}}_\beta(B(x,r))\ge \sum_{y\in B(x,r)} \varphi(y) \ge c_0\cdot \frac{r^d}{\sup_z g_\beta(z,B(x,r))}\ge c\cdot r^\beta,$$ for some positive constants $c_0$ and $c$. The proof of the lower bound then follows from~\eqref{equiv.cap}. 
\end{proof}

The lower bound in the previous result can be generalized as follows. 
\begin{lemma}\label{cap.lowerbound}
There exists $c>0$, such that for any finite $A\subset V$, 
$${\rm Cap}_\beta(A)\ge c \cdot |A|^{\beta/ d}. $$ 
\end{lemma}
\begin{proof} 
Consider $\varphi$ the function which is zero outside $A$, and constant on $A$ equal to $\tfrac{1}{ \sup_{x\in A} g_\beta(x,A)}$. 
A simple rearrangement inequality shows that $\sup_{x\in A} g_\beta(x,A) \le C\cdot |A|^{1 - \beta/d}$, for some constant $C>0$, which concludes the proof, using~\eqref{equiv.cap}. 
\end{proof}

For the next result, whose proof will be given in the appendix, it is convenient to introduce the following notation. For $A\subset V$, and $r>0$, we write
\begin{equation}\label{not.BAr}
B(A,r) = \cup_{x\in A} B(x,r).
\end{equation}
Also, a set $A\subset V$ is said to be {\it $s$-well separated}, if $d_V(x,y)\ge s$, for all $x\neq y\in A$. 

\begin{lemma} \label{lem.cap.extract}
There exists a constant $c>0$, such that for any $r\ge 1$, and any $2r$-well separated finite set $A\subset V$, one can find $U\subseteq A$, such that
$${\rm Cap}_\beta\big(B(U,r)\big) \ge c|U|\cdot r^\beta, \quad \textrm{and} \quad {\rm Cap}_\beta\big(B(U,r)\big)  \ge c \cdot {\rm Cap}_\beta\big( B(A,r)\big)  . $$
\end{lemma}


\section{Proof of Theorem~\ref{thm.expmoment}}\label{sec.proofmoment} 
We follow a similar approach as in the proof  of Theorem 1.5 in~\cite{ASS23}, we integrate the occupation field of an open cluster  against a test function $\varphi$, in order to optimize and unravel a rate function which turns out to be a capacity.

More precisely, we aim to show, that there exists a constant $C>0$, such that for any $n\ge 1$, and any $x\in V$, 
\begin{equation}\label{induction}
\mathbb E\left[\left(\sum_{y\in \mathcal C(x)} \varphi(y)\right)^n\right] \le 
C^{n-1} \cdot (1\vee (n-2))! \cdot g_\alpha *\varphi(x). 
\end{equation} 
Note first that it would prove the theorem, by summation over $n$. Observe now that~\eqref{induction} immediately follows from~\eqref{Halpha} for $n=1$, and also that for some constant $C_0>0$, 
\begin{equation}
    \label{n=1}
\mathbb E\left[\sum_{y\in \mathcal C(x)} \varphi(y)\right] = \sum_{y\in V} \tau(x,y)\varphi(y) = \tau * \varphi(x)\le C_0 \cdot g_\alpha*\varphi(x) \le C_0,
\end{equation}
where the last equality follows from the hypothesis $\|g_{2\alpha-d}*\varphi \|_\infty\le 1$, and the fact that $g_\alpha \le g_{2\alpha-d}$. 
We now prove the induction step. Assume that~\eqref{induction} holds for some $n-1\ge 1$, and let us prove it for $n$. Expanding the $n$-th moment, we get that 
$$\mathbb E\left[\left(\sum_{y\in \mathcal C(x)} \varphi(y)\right)^n\right] = 
\sum_{y_1,\dots,y_n \in V} 
\mathbb P\big(y_1,\dots,y_n \in \mathcal C(x)\big) \cdot \prod_{i=1}^n \varphi(y_i). $$
On the event $\{y_1,\dots,y_n\in \mathcal C(x)\}$
we distinguish two cases. Either there exists $i\in \{1,\dots,n\}$, such that all the points $y_1,\dots,y_n$, can be connected to $y_i$ \textit{disjointly} from a self-avoiding open path connecting $x$ to $y_i$;  
or there must exist a vertex $y\in \mathcal C(x)$ (possibly $x$), and a partition of the points $y_1,\dots,y_n$ into at least $2$ nonempty  subsets, which are connected {\it disjointly} to $y$, and also {\it disjointly} from a path going from $x$ to $y$. In both cases denote by $M(x;y_1,\dots,y_n)$ this vertex $y$ which disconnects $x$ from the other vertices (if there are several choices, one may choose one at random). 
Concerning the first case, one has using BK inequality, 
\begin{align*}
 \mathbb P(y_1,\dots,y_n \in \mathcal C(x), M(x;y_1,\dots,y_n) \in \{y_1,\dots,y_n\})  \le & \sum_{i=1}^n \mathbb P\Big(\{x\longleftrightarrow y_i\}\circ \{y_j\longleftrightarrow y_i, \forall j\neq i\}\Big) \\
& \le \sum_{i=1}^n \tau(x,y_i) \cdot \mathbb P(y_j\longleftrightarrow y_i, \forall j\neq i),  
\end{align*}
and thus by the induction hypothesis, 
\begin{align}\label{nmoment1}
\nonumber & \sum_{y_1,\dots,y_n \in V} 
\mathbb P\Big(y_1,\dots,y_n \in \mathcal C(x),M(x;y_1,\dots,y_n) \in \{y_1,\dots,y_n\}\Big) \cdot \prod_{i=1}^n \varphi(y_i)\\
\nonumber & \le n \cdot \sum_{y\in V} \tau(x,y) \varphi(y) \cdot \mathbb E\left[\left(\sum_{z\in \mathcal C(y)} \varphi(z)\right)^{n-1}\right] \\
& \le 3 C^{n-2}\cdot (n-2)! \cdot g_\alpha* \varphi(x), 
\end{align}
using for the last inequality that $n(1\vee (n-3))! \le 3 (n-2)!$, for all $n\ge 2$. 
Consider now the second case, when $M(x;y_1,\dots,y_n) \notin \{y_1,\dots,y_n\}$.   A union bound gives  
\begin{align*}
& \mathbb P(y_1,\dots,y_n \in \mathcal C(x), M(x;y_1,\dots,y_n) \notin \{y_1,\dots,y_n\}) \\
& \leq \sum_{L= 2}^{n} \sum_{y\in V}
\sum_{\cup_{i=1}^L I_i = \{1,\dots,n\}} \mathbb P\left(\{y\longleftrightarrow x\}\circ \{y_j \in \mathcal C(y), \forall j\in I_1\}\circ \dots \circ \{y_j\in  \mathcal C(y), \forall j\in I_L \}\right), \end{align*}
where the second sum runs over disjoint subsets $I_1,\dots,I_L$, whose union is $\{1,\dots,n\}$. 
We thus obtain using BK inequality at the first line, and the induction hypothesis together with~\eqref{n=1} for the last inequality, 
\begin{align}\label{nmoment}
\nonumber & \sum_{y_1,\dots,y_n \in V} 
\mathbb P\Big(y_1,\dots,y_n \in \mathcal C(x),M(x;y_1,\dots,y_n) \notin \{y_1,\dots,y_n\}\Big) \cdot \prod_{i=1}^n \varphi(y_i) \\
\nonumber & \le \sum_{y\in V} \tau(x,y) \sum_{L=2}^{n} \sum_{\substack{n_1,\dots,n_L\\ \sum_i n_i = n \\  
n_i \ge 1, \ \forall i}} \frac{n!}{n_1!\dots n_L!} \prod_{i=1}^L \mathbb E\left[\left(\sum_{z\in \mathcal C(y)} \varphi(z)\right)^{n_i}\right] \\ 
& \le  n! \cdot \sum_{y\in V} \tau(x,y) \sum_{L=2}^{n}  C^{n-L}C_0^{L-2} \sum_{\substack{n_1,\dots,n_L\\ \sum_i n_i = n \\
n_i \ge 1, \ \forall i}} \prod_{i=1}^L \frac{1}{n_i(n_i-1)} \cdot (g_\alpha*\varphi(y))^2, 
\end{align}
with the convention that $\frac{1}{n_i(n_i-1)} = 1$, when $n_i=1$. 
Now it was observed in~\cite{ASS23} that there exists a constant $c>0$, such that for any $n\ge 2$ and $L\ge 2$, 
\begin{equation}\label{claim1}
\sum_{\substack{n_1,\dots,n_L\\ \sum_i n_i = n \\ n_i \ge 1, \ \forall i}} \prod_{i=1}^L \frac{1}{n_i(n_i-1)}  \le \frac{c^{L-1}}{n^2}.
\end{equation}
Furthermore, using~\eqref{Halpha} and~\eqref{poly} together with the fact that $\|g_{2\alpha -d}*\varphi\|_\infty \le 1$, we get  
\begin{equation}\label{claim2}
\sum_{y\in V} \tau(x,y)\cdot (g_\alpha*\varphi(y))^2 \le c' \cdot g_\alpha*\varphi(x),
\end{equation}
for some constant $c'>0$. 
Inserting~\eqref{claim1} and~\eqref{claim2} in~\eqref{nmoment} and using also~\eqref{nmoment1} yields 
$$\mathbb E\left[\left(\sum_{y\in \mathcal C(x)} \varphi(y)\right)^n\right] \le C^{n-2}\cdot (n-2)! \cdot \Big(3 +\sum_{L\ge 2} \frac{(cC_0)^{L-1}c'}{C^{L-2}}\Big) \cdot g_\alpha*\varphi(x), $$ 
proving well the induction step, provided $C$ is taken large enough. \hfill $\square$

\section{Proof of Theorem~\ref{LD.intersection.finite}}\label{sec.proof.main}
We present here a proof of Theorem~\ref{LD.intersection.finite}. The approach we developed for studying the intersection of the ranges of two independent random walks in~\cite{AS23} relies on viewing the excess deviation  problem as a problem
of random walk in random environment, where the walk and the environment are decoupled and the environment is played by the occupation field of one walk. 
Namely, we decompose
the occupation field of one of the range, or cluster here, intersected with a large ball, say $B(0,R)$, into
(random) regions where on a given space-scale $r$, the occupation density exceeds some value $\rho$. Using features of an appropriate
{\it capacity}, upper bounds on 
the volume of such random regions can be established for all couples $(r,\rho)$ satisfying some constraint that also depends on $R$, see~\eqref{cond.rrho} below. 
Then dealing with the intersection of two independent  clusters, after
we slice the occupation field according to the occupation density of one of them, we transform
the problem into estimating the chance the other cluster occupies
a random region of prescribed volume and density, more than its typical value. Thus, through this approach, we reduce the problem into understanding the covering property of a single process.

Before we enter in the core of the argument, we shall need a number of preparatory lemmas, which we gather in Subsection~\ref{subsec.prelim}. Then we conclude the proof in Subsection~\ref{subsec.proofmain}.  

\subsection{Preliminaries}\label{subsec.prelim} 
Recall the notation from~\eqref{not.BAr}, and the definition following it. 

\begin{proposition}\label{lem.unionballs}
Assume that~\eqref{Halpha} holds, for some 
$\alpha \in (d/2,d)$. There exists a constant $\kappa>0$, such that for any $r\ge 1$, any $\rho >0$, and any $2r$-well separated finite set $A\subset V$, one has 
$$\mathbb P\Big(|\mathcal C_0 \cap B(x,r)|\ge \rho r^d, \textrm{ for all } x \in A\Big) \le \exp\Big(- \kappa \rho\cdot {\rm Cap}_{2\alpha-d}\big( B(A,r)\big)\Big). $$ 
\end{proposition}
\begin{proof}
Consider a function $\varphi$ which realizes the maximum in the definition of $\widetilde{\rm Cap}_{2\alpha-d}\big( B(A,r)\big)$. Define $\widetilde \varphi$ a function which is constant on each ball  $B(x,r)$, with $x \in A$, zero outside these balls, and such that for any $x\in A$, and $y\in B(x,r)$, $\widetilde \varphi(y) = \frac 1{r^d}\sum_{z\in B(x,r)} \varphi(z)$. 
Notice that for some constant $c>0$, one has $\|g_{2\alpha - d}* \widetilde \varphi\|_\infty \le c$. Indeed, given any $z\in V$, one has for some constant $C>0$, 
\begin{align*}
& \sum_{x\in A:d(x,z)\ge 2r} \sum_{y\in B(x,r)} g_{2\alpha-d}(z,y) \widetilde \varphi(y)  \le C \sum_{x\in A:d(x,z)\ge 2r} g_{2\alpha-d}(x,z) \sum_{y\in B(x,r)}  \widetilde \varphi(y) \\
& \le C^2 \sum_{x\in A:d(x,z)\ge 2r} \sum_{y\in B(x,r)} g_{2\alpha-d}(z,y) \varphi(y)  \le C^2,
\end{align*}
where the last inequality follows from the fact that by $\|g_{2\alpha-d}*\varphi \|_\infty \le 1$, by hypothesis. Furthermore, using that $\sum_{y\in B(z,3r)} g_{2\alpha-d}(z,y)\le Cr^{2d-2\alpha}$,  yields   
\begin{align*}
& \sum_{x\in A :d(x,z)\le 2r} \sum_{y\in B(x,r)} g_{2\alpha-d}(z,y) \widetilde \varphi(y) \le C\cdot r^{d-2\alpha} \sum_{x\in A : d(x,z) \le 2r} \sum_{y\in B(x,r)} \varphi(y) \\
& \le C^2 r^{d-2\alpha} \cdot {\rm Cap}_{2\alpha- d}\big(\cup_{x\in A : d(x,z)\le 2r} B(x,r)\big)\le C^3,
\end{align*}
where the second inequality follows from the definition of $\widetilde {{\rm Cap}}_{2\alpha- d}$ and~\eqref{equiv.cap}, and the third one from  Lemmas~\ref{cap.translation}  and~\ref{cap.ball}. 
Altogether, this proves our claim that for some constant $c>0$, one has $\|g_{2\alpha - d}* \widetilde \varphi\|_\infty \le c$. 
Now we observe that 
\begin{align*}
& \mathbb P\Big(|\mathcal C_0 \cap B(x,r)|\ge \rho r^d, \textrm{ for all } x \in A\Big) \le \mathbb P\Big(\frac 1c\sum_{z\in \mathcal C_0} \widetilde \varphi(z) \ge \frac{\rho}{c} \sum_{z\in V} \varphi(z)\Big),
\end{align*}
and thus the result follows from Chebyshev's exponential inequality and Theorem~\ref{thm.expmoment},  together with the definition of $\varphi$, and~\eqref{equiv.cap}.  
\end{proof}

\begin{lemma}\label{lem.densite.1}
There exists a constant $C>0$, such that for any $\rho>0$, any $r\ge 1$, and any finite set $\Lambda\subset V$ satisfying 
$$|\Lambda\cap B(x,r)|\le \rho \cdot |B(x,r)|,\quad \textrm{for all }x\in \Lambda,$$  
one has for any $R\ge r$, and any $x\in V$, 
$$|\Lambda\cap B(x,R)|\le C \cdot \rho R^d. $$
\end{lemma}
\begin{proof}
Let $\Lambda$ be a set satisfying the hypothesis of the lemma for some $r\ge 1$ and $\rho>0$. Fix $x\in V$ and $R\ge r$, and assume that $\Lambda\cap B(x,R)$ is non-empty, as otherwise there is nothing to prove. Then we define inductively a sequence $(x_n)_{n\le n_0}$ in $B(x,R)$ as follows. First pick arbitrarily $x_1\in \Lambda\cap B(x,R)$. Now assuming, $x_1,\dots,x_n$ have been chosen, for some $n\ge 1$, pick arbitrarily $x_{n+1} \in \Lambda\cap B(x,R) \cap (\cup_{k\le n} B(x_k,r))^c $, if this set is nonempty. Otherwise stop the process and set $n_0=n$. Note that $n_0$ cannot be larger than $2^dc_2R^d/(c_1r^d)$, with $c_1$ and $c_2$ as in~\eqref{poly}, for otherwise we would have 
$$|B(x,R)|\ge \sum_{k=1}^{n_0} |B(x_k,r/2)| \ge n_0 c_1(r/2)^d > c_2 R^d,$$
contradicting~\eqref{poly}. Therefore, 
$$|\Lambda \cap B(x,R)| = |\bigcup_{k\le n_0} \Lambda\cap B(x_k,r)| \le c_2n_0 \rho r^d \le C\cdot \rho R^d,$$
with $C = 2^dc_2^2/c_1$.
\end{proof}

For the next result we need some additional notation. For positive $r$, $R$, and $\rho$, let 
$$\mathcal C_0^R(r,\rho) = \big\{x\in \mathcal C_0 : \rho\,  |B(x,r)| \le |\mathcal C_0 \cap B(x,r)\cap B(0,R) |\le 2 \rho\, |B(x,r)|\big\}. $$ 
\begin{proposition}\label{lem.COR1}
Assume that $\mathcal G$ satisfies~\eqref{Halpha}, for some $\alpha\in(d/2,d)$. 
There exist positive constants $C_0$ and $\kappa$, such that for any $r,R\ge 1$ and  $\rho>0$, satisfying 
\begin{equation}\label{cond.rrho}
\rho \cdot r^{2\alpha - d} \ge C_0\cdot \log R, 
\end{equation}
and any $L\ge 1$, one has 
$$\mathbb P(|\mathcal C_0^R(r,\rho)|>L)\le \exp\big(-\kappa \rho^{2(1-\frac{\alpha}{d})}\cdot L^{\frac{2\alpha}{d} -1}\big).
$$
\end{proposition}
\begin{proof} 
We note that on the event $\{|\mathcal C_0^R(r,\rho)|>L\}$, there are at least $n_0=\lfloor \tfrac{L}{2^{d+1}C\rho r^d}\rfloor$ points in $C_0^R(r,\rho)$ which are at distance at least $2r$ one from each other, with $C$ as in Lemma~\ref{lem.densite.1}. To see this, define inductively $(x_n)_{n\le n_0}\in \mathcal C_0^R(r,\rho)$, as follows. First pick arbitrarily $x_1\in \mathcal C_0^R(r,\rho)$. 
Then if $x_1,\dots,x_n$ have been defined for some $n<n_0$, pick arbitrarily $x_{n+1}$ in $(\cup_{i\le n}B(x_i,2r))^c\cap \mathcal C_0^R(r,\rho)$, if the latter is nonempty. 
The claim is that this is always the case until $n=n_0$, which is readily seen by the fact that by Lemma~\ref{lem.densite.1}, for any $n< n_0$, one has $|\cup_{i\le n} \mathcal C_0\cap B(x_i,2r)|\le n_0\cdot C\rho (2r)^d \le  L/2$. Now by  applying Lemma~\ref{lem.cap.extract}, one can extract a subset $U\subset \{x_1,\dots,x_{n_0}\}$, such that for some $c>0$, 
$${\rm Cap}_{2\alpha -d}\big(B(U,r)\big)\ge c \cdot |U|\cdot  r^{2\alpha -d}, \quad \textrm{and} \quad 
{\rm Cap}_{2\alpha -d}\big(B(U,r)\big)\ge c \cdot {\rm Cap}_{2\alpha -d}\big(\cup_{k=1}^{n_0} B(x_k,r)\big). $$ 
Then recall that the points $(x_k)_{k\le n_0}$ are all at distance at least $2r$ one from each other, which implies  
$$\big|\bigcup_{k=1}^{n_0} B(x_k,r)\big|=  \sum_{k=1}^{n_0} |B(x_k,r)| \ge c_1 \cdot n_0 r^d,$$
with $c_1$ as in~\eqref{poly}. Together with  Lemma~\ref{cap.lowerbound}, this yields for some $c>0$,  
$${\rm Cap}_{2\alpha -d}\big(\cup_{k=1}^{n_0} B(x_k,r)\big) \ge c (L/\rho)^{\frac{2\alpha}{d}- 1}.$$
Moreover, given $|U|$, the number of possible choices for the set $U$ is at most $|B(0,R)|^{|U|} \le \exp(c'\cdot |U|\cdot \log R)$, for some $c'>0$. Taking $C_0$ sufficiently large in~\eqref{cond.rrho} ensures that this factor can be absorbed by the exponential decay given by Proposition~\ref{lem.unionballs}. As a result, using a union bound over all possible $U$, gives  that for some constant $\kappa>0$, 
\begin{equation*}
\mathbb P(|\mathcal C_0^R(r,\rho)|>L)\le  \exp\big( - \kappa \rho \cdot (L/\rho)^{\frac{2\alpha}{d}- 1}\big) =  \exp\big(-\kappa\cdot  \rho^{2(1-\frac{\alpha}{d}) }\cdot L^{\frac{2\alpha}{d}-1}\big), 
\end{equation*}
as wanted. 
\end{proof}
Define now 
$$\mathcal C_0^{R,*}(r,\rho) = \big\{x\in \mathcal C_0 :  |\mathcal C_0 \cap B(x,r)\cap B(0,R) |\ge  \rho \, |B(x,r)|\big\}. $$ 
 
\begin{corollary}\label{cor.COR}
Assume that $\mathcal G$ satisfies~\eqref{Halpha}, for some $\alpha \in (d/2,d)$. There exist positive constants $C$, $C_0$ and $\kappa$, such that for any $r,R\ge 1$ and  $\rho>0$, satisfying~\eqref{cond.rrho}, 
and any $L\ge 1$, one has 
$$\mathbb P(|\mathcal C_0^{R,*}(r,\rho)|>L)\le C\cdot \exp\big(-\kappa \rho^{2(1-\frac{\alpha}{d})}\cdot L^{\frac{2\alpha}{d} -1}\big).
$$
\end{corollary}
\begin{proof} Let $\varepsilon = \tfrac{d-\alpha}{2\alpha-d}$, and $c_\varepsilon = 1/(\sum_{i\ge 0} 2^{-\varepsilon i})$. Note that $\mathcal C_0^{R,*}(r,\rho) = \cup_{i\ge 0} \mathcal C_0^R(r,2^i\rho)$, and thus by a union bound, and Proposition~\ref{lem.COR1}, 
\begin{align*}
\mathbb P(|\mathcal C_0^{R,*}(r,\rho)|>L) & \le 
\sum_{i\ge 0} \mathbb P(|\mathcal C_0^R(r,2^i\rho)|>c_\varepsilon \cdot \frac{L}{2^{\varepsilon i}}) \\
& \le \sum_{i\ge 0} \exp\big(-\kappa\, c_\varepsilon \cdot 2^{i(1-\frac{\alpha}{d})}\cdot   \rho^{2(1-\frac{\alpha}{d})}\cdot L^{\frac{2\alpha}{d} -1}\big)\\
& \le C\cdot \exp\big(-\kappa'  \rho^{2(1-\frac{\alpha}{d})}\cdot L^{\frac{2\alpha}{d} -1}\big),
\end{align*}
for some constants $\kappa,\kappa'>0$. 
\end{proof}
One last ingredient before we proceed to the proof of Theorem~\ref{LD.intersection.finite} is the following elementary fact. 
\begin{lemma}\label{lem.densite}
Fix $\beta\in (0,d)$. There exists a constant $C>0$, such that the following holds, for any $\rho>0$ and $r\ge 1$. Let $\Lambda\subset V$ be a finite set, such that $$|\Lambda\cap B(x,r)|\le \rho \cdot |B(x,r)|,\quad \textrm{for all }x\in \Lambda.$$  
Then for any $x\in V$, 
$$g_\beta(x,\Lambda\cap B(x,r)^c)\le C \cdot \rho^{\beta/d}\cdot |\Lambda|^{1-\beta/d}. $$ 
\end{lemma}
\begin{proof}
Set $\mathcal S_k = B(x,(k+1)r)\setminus B(x,kr)$, for $k\ge 1$. By Lemma~\ref{lem.densite.1}, one has for some constant $C>0$, 
\begin{align*}
 g_\beta(x,\Lambda\cap B(x,r)^c) & \le \sum_{k\ge 1} g_\beta(x,\Lambda \cap \mathcal S_k) 
\le \sum_{k\ge 1} \frac{1}{(kr)^\beta} \cdot |\Lambda\cap \mathcal S_k| \\
& \le C r^{-\beta} \sum_{k\ge 1} \frac {|\Lambda \cap B(x,kr)|}{k^{\beta+1}} \le Cr^{-\beta} \sum_{k\ge 1} \frac{\min(\rho (kr)^d, |\Lambda|)}{k^{\beta+1}}\\
& \le C\cdot \rho^{\beta/d} \cdot |\Lambda|^{1- \beta/d}. 
\end{align*}
\end{proof}

\subsection{Proof of Theorem~\ref{LD.intersection.finite}} \label{subsec.proofmain} 
First notice that for any $r>1$, 
$$\mathbb P(\widetilde {\mathcal C}_0\cap \mathcal C_0 \cap B(0,r)^c \neq \emptyset) 
\le \sum_{x \in B(0,r)^c} \tau(0,x)^2 
\le C\cdot \sum_{x \in B(0,r)^c} g_\alpha(0,x)^2 
\le C\cdot r^{d-2\alpha},$$
where $C$ is positive constant. 
If we tale $r$ equal to $R_t=\exp(t^{\frac{2\alpha}d - 1})$, the probability above is negligible. 
Thus, if we define 
$$\mathcal C_t(0) = \mathcal C_0\cap B(0,R_t),$$ 
it just remains to prove that for some positive constants $c$ and $C$, 
\begin{equation}\label{goal.intersection}
\mathbb P\Big(\sup_{x\in V}\ g_{2\alpha -d}(x,\mathcal C_t(0)) > C \cdot t^{2- \frac{2\alpha}{d}}\Big) 
\le C \cdot \exp(-c\cdot t^{\frac{2\alpha}{d}-1}). 
\end{equation}
Indeed, once this is proved, then Theorem~\ref{LD.intersection.finite} follows from Corollary~\ref{cor.1} applied to the cluster $\widetilde {\mathcal C}_0$ and $A = \mathcal C_t(0)$.

To prove~\eqref{goal.intersection} we introduce a sequence of scales and densities, linked by the constraint~\eqref{cond.rrho}, which they should satisfy. For $i\ge 0$, set $\rho_i = 2^{-i}$, and define $r_i$ by
$$\rho_i\cdot  r_i^{2\alpha - d} = C_0 \cdot t^{\tfrac{2\alpha}{d}-1},$$
with $C_0$ as in~\eqref{cond.rrho}. Let also 
\begin{equation}\label{setsLambdai}
\Lambda_i = \mathcal C_0^{R_t,*}(r_i,\rho_i)\setminus \left(\bigcup_{0\le j <i} \mathcal C_0^{R_t,*}(r_j,\rho_j)\right), \quad \textrm{and}\quad \Lambda_i^* = \mathcal C_t(0)\setminus \left(\bigcup_{0\le j \le i} \mathcal C_0^{R_t,*}(r_j,\rho_j)\right). 
\end{equation}
By Corollary~\ref{cor.COR}, one has for all $i\ge 0$, 
$$\mathbb P\big(|\Lambda_i|>\rho_i^{\frac{2\alpha - 2d}{2\alpha -d}} \cdot t\big)\le C\cdot \exp\big(-\kappa  \cdot t^{\frac {2\alpha} d - 1}\big). $$ 
Moreover, as soon as $\rho_i r_i^d>c_2R_t^d$, with $c_2$ as in~\eqref{poly}, then $\Lambda_i$ is empty, in particular this holds when $r_i^{2d-2\alpha} > c_2R_t^d/C_0$. Thus for $i>C\cdot \log R_t$, and $C$ some large enough constant, one has $\Lambda_i=\emptyset$. As a consequence, denoting by $\mathcal E$ the event 
\begin{equation}\label{def.E}
\mathcal E = \bigcap_{i\ge 0} \Big\{|\Lambda_i|\le \rho_i^{\frac{2\alpha - 2d}{2\alpha -d}}\cdot t\Big\}, 
\end{equation}
one has by a union bound, and for possibly smaller $\kappa$ and larger $C$, 
$$\mathbb P(\mathcal E^c) \le C\cdot \exp\big(-\kappa  \cdot t^{\frac {2\alpha} d - 1}\big). $$
On the other hand, we claim that on the event $\mathcal E$, one has 
\begin{equation}\label{claim.gCt0}
\sup_{x\in V}\ g_{2\alpha -d}(x,\mathcal C_t(0)) \le  C \cdot t^{2- \frac{2\alpha}{d}},
\end{equation}
which would conclude the proof. So let us prove this claim. Fix $x\in V$, and for $k\ge 1$, write $\mathcal S_k = B(x,r_k) \setminus B(x,r_{k-1})$. Let also $\mathcal S_0 = B(x,r_0)$, so that these sets $(\mathcal S_k)_{k\ge 0}$ form a partition of $V$. Note that, for some constant $C>0$, that might change from line to line,  
$$g_{2\alpha-d}(x,\mathcal C_t(0) \cap \mathcal S_0)\le g_{2\alpha-d}(x,\mathcal S_0)
\le C r_0^{2d-2\alpha} \le C \cdot t^{2-\frac{2\alpha}{d}}.$$ 
Furthermore, for any $k\ge 1$, 
$$g_{2\alpha-d}(x,\mathcal C_t(0) \cap \mathcal S_k)= \sum_{i=0}^k g_{2\alpha-d}(x,\Lambda_i \cap \mathcal S_k) + g_{2\alpha-d}(x,\Lambda_k^* \cap \mathcal S_k).$$
Letting $\mathcal S_i^* = \cup_{k \ge i} \mathcal S_k$, and summing over $k\ge 1$, we get 
\begin{equation}\label{sumk} 
\sum_{k\ge 1} 
g_{2\alpha-d}(x,\mathcal C_t(0) \cap \mathcal S_k) = g_{2\alpha-d}(x,\Lambda_0 \cap \mathcal S_1^*) + \sum_{i=1}^\infty  g_{2\alpha-d}(x,\Lambda_i \cap \mathcal S_i^*) + \sum_{k\ge 1} g_{2\alpha-d}(x,\Lambda_k^* \cap \mathcal S_k).
\end{equation}
Concerning the second sum, note that for any $k\ge 1$, 
$$g_{2\alpha-d}(x,\Lambda_k^* \cap \mathcal S_k)\le \frac{|\Lambda_k^* \cap \mathcal S_k|}{r_{k-1}^{2\alpha - d}} \le C\cdot \frac{\rho_k \cdot r_k^d}{r_{k-1}^{2\alpha-d}} \le C\cdot \frac{t^{\frac{2\alpha}{d}-1}}{r_k^{4\alpha - 3d}} .$$ 
Hence, since $\alpha>3d/4$ by hypothesis, 
$$\sum_{k\ge 1}g_{2\alpha-d}(x,\Lambda_k^* \cap \mathcal S_k)\le C\cdot  \frac{t^{\frac{2\alpha}{d}-1}}{r_0^{4\alpha- 3d}} \le C\cdot t^{\frac{2\alpha}{d}-1 - \frac{4\alpha}{d} + 3}= C\cdot t^{2- \frac{2\alpha}{d}} .$$ 
As for the first sum in~\eqref{sumk}, Lemma~\ref{lem.densite} shows that for any $i\ge 1$, on the event $\mathcal E$, 
$$g_{2\alpha-d}(x,\Lambda_i \cap \mathcal S_i^*) \le 
C \rho_i^{\frac{2\alpha}{d}-1} \cdot |\Lambda_i|^{2-\frac{2\alpha}{d}} \le C  t^{2- \frac{2\alpha}{d}}\cdot \rho_i^{\frac{2\alpha}{d}- 1 - \frac{(2\alpha- 2d)^2}{d(2\alpha -d)}}\le C\cdot t^{2- \frac{2\alpha}{d}}\cdot \rho_i^{ \frac{4\alpha +5d}{2\alpha-d}}.$$
Summing now over $i\ge 1$, yields 
$$\sum_{i\ge 1}g_{2\alpha-d}(x,\Lambda_i \cap \mathcal S_i^*) \le C\cdot t^{2-\frac{2\alpha}{d}}. $$
Finally, on the event $\mathcal E$, 
$$g_{2\alpha-d}(x,\Lambda_0 \cap \mathcal S_1^*)
\le \frac{|\Lambda_0|}{r_0^{2\alpha - d}} \le 
\frac{t}{r_0^{2\alpha-d}}\le C \cdot t^{2-\frac{2\alpha}{d}},$$
which altogether proves~\eqref{claim.gCt0} and concludes the proof of Theorem~\ref{LD.intersection.finite}. 
\hfill $\square$

\begin{remark}\label{rem.alphabeta}
\emph{Note that the same argument can be used to  show~\eqref{intersection.twodifferent} concerning the case of clusters of different types. More precisely, assume that $\mathcal G$ and $\widetilde{\mathcal G}$ satisfy respectively~\eqref{Halpha} and $(\mathcal H_\beta)$, for some $\alpha,\beta\in(d/2,d)$, not necessarily equal, with say $\alpha\le \beta$. Then one can define the same sets $\Lambda_i$ for $\mathcal C_0$ as above, let $R_t = \exp(t^\gamma)$, with $\gamma$ as in~\eqref{intersection.twodifferent}, and consider $$\mathcal E = \big\{|\Lambda_i|\le \rho_i^{\frac{2\alpha - 2d}{2\alpha -d}}\cdot t^{\frac{\gamma d}{2\alpha -d}}, \ \forall i\big\}.$$ 
One can check first that $\mathbb P(\mathcal E^c) \le C \exp(-t^\gamma)$, and secondly that on $\mathcal E$, and under the additional assumption that $2\alpha + 2\beta >3d$, one has $\sup_{x\in V} g_{2\beta-d}(x,\mathcal C_t(0))\le Ct^{1-\gamma}$,  using exactly the same argument as above, and this proves~\eqref{intersection.twodifferent}. 
We note that the condition $2\alpha + 2\beta>3d$, should also correspond to the upper critical dimensions in general, in the sense that if 
$2\alpha + 2\beta<3d$, the tail should be polynomial. In the critical case $2\alpha + 2\beta = 3d$, the tail should be stretched exponential, but with an exponent that is unclear. }
\end{remark}

\section{Proof of Theorem~\ref{thm.IIC}} \label{sec.IIC}
Let us start with the proof of~\eqref{IIC.1}. For $p\in [0,1]$, let
$\mathbb P_p$ denotes the law of Bernoulli bond percolation on $\mathbb Z^d$, and let 
$\tau_p(x,y) = \mathbb P_p(x\longleftrightarrow y)$. Let also $p_c$ be the critical value for percolation. It is proved in~\cite{HvdHH}, that for any local functional $F$, 
$$
\lim_{\|z\|\to +\infty} \frac 1{\tau_{p_c}(0,z)}\, \mathbb E_{p_c}[F\cdot \mathbf 1_{\{0\longleftrightarrow z\}}] = \mathbb E_{\textrm{IIC}}[F],
$$
where $\mathbb E_{\textrm{IIC}}$ denotes expectation with respect to the IIC. 
Let now $\varphi:\mathbb Z^d\to [0,\infty)$, be such that $\|g_{d-4}*\varphi\|_\infty\le 1$. Also for $R>0$, let  $x\stackrel{B(0,R)}{\longleftrightarrow} 0$,
be the event that $x$ and $0$ are connected in $B(0,R)$ (in particular this requires $x\in B(0,R)$).
Then $\sum_{x\in \mathbb Z^d} \mathbf 1\{x\stackrel{B(0,R)}{\longleftrightarrow} 0\}\cdot  \varphi(x)$ is a local functional, and thus for any fixed $R>0$, and $\kappa>0$,
\begin{equation*}
\lim_{\|z\|\to +\infty} \frac 1{\tau_{p_c}(0,z)}\mathbb E_{p_c}\left[\exp\Big(\kappa \sum_{x\stackrel{B(0,R)}{\longleftrightarrow} 0} \varphi(x)\Big)\cdot \mathbf 1\{0\longleftrightarrow z\} \right] = \mathbb E_{\textrm{IIC}}\Big[\exp\Big(\kappa \sum_{x\stackrel{B(0,R)}{\longleftrightarrow} 0} \varphi_R(x)\Big)\Big]. 
\end{equation*}
We will show that there exist positive constants $C$ and $\kappa$, such that for any $R>0$, 
\begin{equation}\label{IIC.goal}
\limsup_{\|z\|\to +\infty} \frac 1{\tau_{p_c}(0,z)}\mathbb E_{p_c}\left[\exp\Big(\kappa \sum_{x\stackrel{B(0,R)}{\longleftrightarrow} 0} \varphi(x)\Big)\cdot \mathbf 1\{0\longleftrightarrow z\} \right] \le C\cdot \|\varphi\|_1,
\end{equation}
which will conclude the proof of~\eqref{IIC.1}, since by monotone convergence one has 
$$\lim_{R\to \infty} \mathbb E_{\textrm{IIC}}\Big[\exp\Big(\kappa \sum_{x\stackrel{B(0,R)}{\longleftrightarrow} 0} \varphi(x)\Big)\Big] = \mathbb E\Big[\exp\Big(\kappa \sum_{x\in \mathcal C_\infty} \varphi(x)\Big)\Big].$$
To prove~\eqref{IIC.goal}, notice that by~\eqref{induction}, one has for any $n\ge 1$, and any $z\in \mathbb Z^d$, with $ \varphi_R(x) = \varphi(x)\cdot  \mathbf 1_{\{\|x\|\le R\}}$, 
\begin{align*}\mathbb E_{p_c}\left[\Big(\sum_{x\stackrel{B(0,R)}{\longleftrightarrow} 0} \varphi(x) \Big)^n \mathbf 1\{0\longleftrightarrow z\}\right] 
& \le  \mathbb E_{p_c}\left[\Big(\sum_{x\in \mathcal C_0} \varphi_R(x) \Big)^n \mathbf 1\{0\longleftrightarrow z\} \right]\\
& \le \mathbb E_{p_c}\left[\Big(\sum_{x\in \mathcal C(z)} \varphi_R(x) \Big)^n  \right]\\
& \le C^{n-1} \cdot (1\vee (n-2)!)\cdot g_{d-2}*\varphi_R(z). 
\end{align*}
Note now that for any $z$ with $\|z\|\ge 2 R$, one has for some constant $C_0>0$,  
$$g_{d-2} * \varphi_R(z) \le 2^{d-2}\,  g_{d-2}(0,z) \cdot \|\varphi_R\|_1 \le C_0\cdot  \tau_{p_c}(0,z)\cdot \|\varphi\|_1,$$
using the results of~\cite{FvdH17} on the two-point function for the last inequality. 
Then \eqref{IIC.goal} follows, and this concludes the proof of~\eqref{IIC.1}.

Now with~\eqref{IIC.1} at hand, one can follow the same argument as for the proof of Theorem~\ref{LD.intersection.finite}, to deduce~\eqref{IIC.2}. Indeed, the only place where~\eqref{IIC.1} is used is in the proof of Lemma~\ref{lem.unionballs}, and there one can use that for $\varphi$ realizing the maximum in the definition of $\widetilde{\rm Cap}_{d-4}(B(A,r))$, one has by definition $\|\varphi \|_1 = \widetilde {\rm Cap}_{d-4}(B(A,r))$, which can be absorbed in the exponential bound, at the cost of an additional factor $1/\rho$ in front. As a consequence, one can deduce an analogue of Corollary~\ref{cor.COR} with a factor $1/\rho$ in front of the exponential in the upper bound. More precisely we obtain, with hopefully obvious notation, that for any $r,R\ge 1$ and $\rho>0$ satisfying~\eqref{cond.rrho}, and any $L\ge 1$, 
$$\mathbb P(|\mathcal C_\infty^{R,*}(r,\rho)|>L) \le \frac{C}{\rho} \cdot \exp(-\kappa \rho^{4/d} \cdot L^{1- 4/d}),$$
for some positive constants $C$, $C_0$ and $\kappa$. Then the rest of the proof can be followed word by word, simply taking now $R_t = \exp(\varepsilon t^{1-4/d})$, with $\varepsilon$ a small enough constant, so that the number of sets $\Lambda_i$ to be considered is not too large. This concludes the proof of~\eqref{IIC.2}. \hfill $\square$


\section{Proof of Theorem~\ref{theo.generalcluster}}\label{sec.manyclusters} 
We prove here Theorem~\ref{theo.generalcluster} concerning the extension of our main result to the
intersection of more than two clusters. For this we prove first some general upper bound on the $n$-point function, see Proposition~\ref{prop.npoint} below, which is a key  result of the paper. As  mentioned in the introduction, a similar result could be deduced from the tree-graph inequality of Aizenman and Newman~\cite{AN84}, after summing the space-variables associated to the so-called {\it internal vertices}.  However, we take here a different route, more suited to our applications, and to more general settings such as branching random walks. 
To state the result we need additional definition.

\bigskip 
We consider here finite plane trees, and put a mark on vertices which have either zero or one child. For $n\ge 1$, we denote by $\mathbb U_n$ the set of finite plane trees,  which have $n$ marked vertices. Then denote by $ \mathbb T_n$ the set of trees in $\mathbb U_n$, whose marked vertices are further endowed with a label in $\{1,\dots,n\}$, i.e. a one-to-one map $\varphi$ from $\{1,\dots,n\}$ to the set of marked vertices of the tree. Note that once such $\varphi$ is given, all other labelling are obtained by composing it with an element of $\mathfrak S_n$ the group of permutations of $\{1,\dots,n\}$. In particular $\mathbb T_n$ is in bijection with $\mathbb U_n\times \mathfrak S_n$. 

\bigskip 
We denote by $\mathbb T_n^\bullet$ the set of trees in~$\mathbb T_n$, whose root is a marked vertex, and let $\mathbb T_n^\circ = \mathbb T_n\setminus \mathbb T^\bullet_n $. 
Note that the root of any $\mathfrak t\in \mathbb T_n^\circ$ has $L\ge 2$ children, which are the roots of subtrees $\mathfrak t_1,\dots,\mathfrak t_L$, all belonging to some $\mathbb T_k$, (after some trivial relabelling) for some $k\le n-1$: indeed all marked vertices cannot be in the same subtree, as on one hand $\mathfrak t$ is finite and on the other hand all its leaves are marked. 
Moreover, given $\mathfrak t\in \mathbb T_n$, we define $x_{\mathfrak t}$ as being equal to $x_i$, if $i$ is the label of the first marked vertex of $\mathfrak t$ 
in the lexicographical order. 

\vspace{0.2cm}
We now consider the set $\mathbb T_n^*$ of unordered rooted trees with $n$ vertices with distinct labels in $\{1,\dots,n\}$. We then define  inductively a map $\pi: \mathbb T_n \to \mathbb T_n^*$, as follows. 
\begin{itemize}
\item If $\mathfrak t\in \mathbb T_1$, then necessarily $\mathfrak t$ has only one vertex with label $1$. In this case $\pi(\mathfrak t) = \mathfrak t$. Assume now that $\pi(\mathfrak t)$ has been defined for all $\mathfrak t\in \mathbb T_k$, with $k\le n-1$, and 
consider $\mathfrak t\in \mathbb T_n$, with $n\ge 2$. 
\item If $\mathfrak t \in \mathbb T_n^\bullet$, i.e. if the root of $\mathfrak t$ is a marked vertex, then by definition it has only one child, which is the root of another tree $\mathfrak t'$. Then define $\pi(\mathfrak t)$ as being the tree whose root has the same label as the root of $\mathfrak t$, and also only one child, say $v$, to which we attach  $\pi(\mathfrak t')$ (identifying the root of $\pi(\mathfrak t')$ with $v$). Since $\mathfrak t' \in \mathbb T_{n-1}$, this is well defined by the induction hypothesis. 
\item 
If $\mathfrak t\in \mathbb T_n^\circ$, 
$\pi(\mathfrak t)$ is obtained by considering the $L$ children of the root $v_1,\dots,v_L$, together with the trees $\mathfrak t_1,\dots,\mathfrak t_L$ emanating from these vertices. We join $v_i$ and $v_{i+1}$ by an edge for 
all $i\le L-1$, and for $i\le L$, further 
attach to $v_i$ the tree $\pi(\mathfrak t_i)$ (identifying the root of $\pi(\mathfrak t_i)$ with $v_i$). Finally we declare 
$v_1$ to be the root of $\pi(\mathfrak t)$.  
\end{itemize}
Note that in our construction, for any $\mathfrak t\in \mathbb T_n$, the root of $\pi(\mathfrak t)$ has the same label as the first marked vertex of $\mathfrak t$. 
We are now in position to define precisely the function $G_\alpha$, which has been  already briefly introduced in~\eqref{npoint}. First, given $\mathfrak t\in \mathbb T_n^*$, we let $\mathcal E(\mathfrak t)$ denote its set of edges, and we identify the vertices with their label. In particular $\{i,j\} \in \mathcal E(\mathfrak t)$ if the vertices with label $i$ and $j$ are joined by an edge. Then given $\underline x=(x_1,\dots,x_n)$, and $\mathfrak t \in \mathbb T_n$, we let 
\begin{equation}\label{def.Galpha}
G_\alpha(\mathfrak t,\underline x) = \prod_{\{i,j\}\in \mathcal E(\pi(\mathfrak t))} g_{2\alpha- d} (x_i,x_j),
\end{equation}
with the convention that if $\mathfrak t$ has only one vertex, then $G_\alpha(\mathfrak t,x_1) = 1$. Note that the only important aspect of $\pi(\mathfrak t)$ in  the definition above is its tree structure, which is why we deal with unordered trees. 
We can now state our main bound on the $n$-point function.

\begin{proposition}\label{prop.npoint}
Assume that $\mathcal G$ satisfies the BK inequality and \eqref{Halpha}, for some $\alpha \in (d/2,d)$. There exists a constant $C>0$, such that for any $x_1,\dots,x_n,z\in V$, one has with $\underline x=(x_1,\dots,x_n)$, 
$$\mathbb P(x_1,\dots,x_n \in \mathcal C(z)) \le C^n \sum_{\mathfrak t\in \mathbb T_n} G_\alpha (\mathfrak t,\underline x) \cdot g_\alpha (x_{\mathfrak t},z). $$ 
\end{proposition}
\begin{proof}
The proof goes by induction on $n$. If $n=1$, the result directly follows from the hypothesis~\eqref{Halpha}, since $\mathbb T_1$ is reduced to the unique tree $\mathfrak t$ with only one vertex labelled $1$, and in this case $G_\alpha(\mathfrak t,x_1) = 1$ by definition. Assume now that the result holds true for any $k\le n-1$, and consider $x_1,\dots,x_n,z \in V$. Following the same argument (and notation) as in the proof of Theorem~\ref{thm.expmoment}, one has 
$$\mathbb P(x_1,\dots,x_n \in \mathcal C(z), M(z;x_1,\dots,x_n) \in \{x_1,\dots,x_n\}) \le C\sum_{i=1}^n g_\alpha(x_i,z) \cdot \mathbb P(x_j \in \mathcal C(x_i), \ \forall j\neq i). $$ 
Then the induction hypothesis yields 
\begin{align*}
\mathbb P(x_1,\dots,x_n \in \mathcal C(z), M(z;x_1,\dots,x_n) & \in \{x_1,\dots,x_n\}) \\
& \le C^n\sum_{i=1}^n g_\alpha(x_i,z) \sum_{\mathfrak t\in \mathbb T_{n-1}}
G_\alpha(\mathfrak t,(x_j)_{j\neq i})\cdot g_\alpha(x_{\mathfrak t},x_i). 
\end{align*}
Now as we already saw, given $\mathfrak t\in \mathbb T_{n-1}$, and $i\in \{1,\dots,n\}$, one can define a new tree $\mathfrak t' \in \mathbb T_n^\bullet$, by considering a root with label $i$ and a unique child to which we attach the tree $\mathfrak t$ (with labels larger than or equal to $i$ increased by one unit). Note that this induces a bijection between $\{1,\dots,n\}\times \mathbb T_{n-1}$ and $\mathbb T_n^\bullet$. 
Bounding also the term $g_\alpha(x_{\mathfrak t},x_i)$ by $g_{2\alpha-d}(x_\mathfrak t,x_i)$, and remembering the inductive definition of $G_\alpha$, we readily get
\begin{equation}\label{eq.propnpoint1}
\mathbb P(x_1,\dots,x_n \in \mathcal C(z), M(z;x_1,\dots,x_n) \in \{x_1,\dots,x_n\}) \le
C^n \sum_{\mathfrak t \in \mathbb T_n^\bullet} G_\alpha(\mathfrak t, \underline x) \cdot g_\alpha(x_{\mathfrak t},z). 
\end{equation}
On the other hand, using again the induction hypothesis, 
we get 
\begin{align*}
 \mathbb P(x_1,\dots,x_n \in \mathcal C(z), & M(z;x_1,\dots,x_n)  \notin \{x_1,\dots,x_n\}) \\
 & \le C \sum_{L=2}^n \sum_{y\in V} \sum_{I_1,\dots,I_L} g_\alpha(z,y) \prod_{\ell = 1}^L \mathbb P\big((x_i)_{i\in I_\ell} \in \mathcal C(y)\big) \\
& \le C^n \sum_{L=2}^n \sum_{y\in V} \sum_{I_1,\dots,I_L} \sum_{\mathfrak t_1,\dots,\mathfrak t_\ell} g_\alpha(z,y) \prod_{\ell = 1}^L G_\alpha(\mathfrak t_\ell, (x_i)_{i\in I_\ell}) \cdot g_\alpha(x_{\mathfrak t_\ell},y), 
\end{align*}
where the sum over $I_1,\dots,I_L$, is over partitions of $\{1,\dots,n\}$, and the last sum is over trees $\mathfrak t_\ell \in \mathbb T_{|I_\ell|}$, for $1 \le \ell \le L$. Then applying Lemma~\ref{lem.npoint} gives 
\begin{align*}
&  \mathbb P(x_1,\dots,x_n \in \mathcal C(z), M(z;x_1,\dots,x_n)  \notin \{x_1,\dots,x_n\}) \\
&  \le C^n \sum_{L=2}^n  \sum_{I_1,\dots,I_L} \sum_{\mathfrak t_1,\dots,\mathfrak t_\ell}\sum_{\sigma \in \mathfrak S_L} g_\alpha(z,x_{\mathfrak t_{\sigma(1)}})\cdot \Big(\prod_{\ell = 1}^{L-1} g_{2\alpha-d}(x_{\mathfrak t_{\sigma(\ell)}},x_{\mathfrak t_{\sigma(\ell+1)}})\Big)\cdot  \prod_{\ell = 1}^L G_\alpha(\mathfrak t_\ell, (x_i)_{i\in I_\ell}). 
\end{align*}
Now observe that given any $L\ge 2$, any partition $I_1,\dots,I_L$ of $\{1,\dots,n\}$, any trees $\mathfrak t_1\in \mathbb T_{|I_1|}, \dots,
\mathfrak t_L \in \mathbb T_{|I_L|}$, and any $\sigma\in \mathfrak S_L$, one can construct a tree $\mathfrak t\in \mathbb T_n^\circ$, by considering an unmarked root with $L$ children, and then for each $1\le \ell\le L$,  attach to the $\ell$-th child the tree $\mathfrak t_{\sigma(\ell)}$ (where for each $k\le |I_\ell|$, we replace the label $k$ by the $k$-th smallest integer from $I_\ell$). Note that this construction induces a bijection between $\mathbb T_n^\circ$ and the set of elements  $(L,I_1,\dots,I_L,\mathfrak t_1,\dots,\mathfrak t_L,\sigma)$ as before. Remembering also the inductive definition of $G_\alpha$, this readily proves that 
\begin{equation}\label{eq.propnpoint2}
\mathbb P(x_1,\dots,x_n \in \mathcal C(z), M(z;x_1,\dots,x_n)  \notin \{x_1,\dots,x_n\})\le C^n \sum_{\mathfrak t\in \mathbb T_n^\circ} G_\alpha(\mathfrak t,\underline x)\cdot g_\alpha(x_{\mathfrak t},z). \end{equation}
Combining~\eqref{eq.propnpoint1} and~\eqref{eq.propnpoint2} concludes the proof of the induction step, hence of the proposition. 
\end{proof}

We shall also need the following lemma. 

\begin{lemma}\label{lem.cardinalTn}
There exists a constant $C>0$, such that for any $n\ge 1$, $|\mathbb U_n|\le C^n$, and $|\mathbb T_n|\le C^n \, n!$.  
\end{lemma}
\begin{proof}
Recall that $\mathbb T_n$ is in bijection with $\mathbb U_n\times \mathfrak S_n$, thus we just need to prove the result for $\mathbb U_n$. Let us prove by induction on $n\ge 1$, that $|\mathbb U_n|\le \tfrac{C^{n-1}}{n^2}$, for some constant $C>0$. The result for $n=1$ is immediate, since $\mathbb U_1$ is reduced to a unique tree, with only one root. Assume now that it holds true for all $k\le n$, and let us prove it for $n+1$. Let $\mathbb U_n^\bullet$ be the set of trees in $\mathbb U_n$ where the root is a marked vertex, and $\mathbb U_n^\circ= \mathbb U_n\setminus \mathbb U_n^\bullet$. Recall that $\mathbb U_{n+1}^\bullet$ is in bijection with $\mathbb U_n$. On the other hand, one has using the induction hypothesis and~\eqref{claim1}, 
\begin{align*}
|\mathbb U_{n+1}^\circ|&  = \sum_{L\ge 2} \sum_{\substack{n_1+\dots+ n_L = n+1\\n_i\ge 1\, \forall i}} |\mathbb U_{n_1}|\times \dots\times |\mathbb U_{n_L}| \le C^{n+1-L} \sum_{L\ge 2} \sum_{\substack{n_1+\dots+ n_L = n+1\\n_i\ge 1\, \forall i}} \prod_{i=1}^L \frac 1{n_i^2} \\
& \le \frac{c\, C^{n-1}}{n^2}  \sum_{L\ge 2} (c/C)^{L-2} \le \frac{c'\, C^{n-1}}{n^2}, 
\end{align*}
for some constants $c$ and $c'$, which we can always assume to be smaller than $C/2$. Combining the last inequality with the previous bound on $|\mathbb U_n^\bullet|$ concludes the proof of the induction step, hence of the lemma. 
\end{proof}

We can now prove the following generalization of Theorem~\ref{thm.expmoment}. 

\begin{theorem}\label{theo.generalmoment}
Assume that $\mathcal G$ satisfies BK inequality and~\eqref{Halpha}, for some $\alpha \in (d/2,d)$. Let $k\ge 1$, and $(\mathcal C_0^i)_{1\le i\le k}$, be independent copies of $\mathcal C_0$.  
There exists $\kappa>0$, such that for any $\varphi:\mathbb Z^d\to [0,\infty)$, satisfying $\|g_{k(2\alpha-d)} *\varphi \|_\infty \le 1$, 
$$\mathbb E\left[\exp \Big(\kappa \sum_{x\in \mathcal C_0^1 \cap \dots \cap \mathcal C_0^k} \varphi(x)\Big)^{1/k}\right] \le 2. $$ 
\end{theorem}
\begin{proof}
Using Proposition~\ref{prop.npoint}, we get for some constant $C>0$, and for any $n\ge 1$, 
\begin{align*} 
\mathbb E\Big[\Big(\sum_{x\in \mathcal C_0^1 \cap \dots \cap \mathcal C_0^k} \varphi(x)\Big)^n\Big]& = 
\sum_{x_1,\dots,x_n} \Big(\prod_{i=1}^n \varphi(x_i)\Big) \cdot \mathbb P(x_1,\dots,x_n\in \mathcal C_0)^k\\
& \le (C^k)^n  \sum_{x_1,\dots,x_n} \Big(\prod_{i=1}^n \varphi(x_i)\Big) \cdot \Big(\sum_{\mathfrak t\in \mathbb T_n} G_\alpha(\mathfrak t,\underline x)\cdot  g_\alpha(x_{\mathfrak t},0)\Big)^k\\
& \le  (C^k)^n |\mathbb T_n|^{k-1}  \sum_{x_1,\dots,x_n} \Big(\prod_{i=1}^n \varphi(x_i)\Big) \cdot \sum_{\mathfrak t\in \mathbb T_n} G_\alpha(\mathfrak t,\underline x)^k \\
& \le (C^{2k})^n (n!)^{k-1}  \sum_{x_1,\dots,x_n} \Big(\prod_{i=1}^n \varphi(x_i)\Big) \cdot \sum_{\mathfrak t\in \mathbb T_n} \prod_{\{i,j\}\in \mathcal E(\pi(\mathfrak t))} g_{k(2\alpha-d)}(x_i,x_j),
\end{align*}
using the definition~\eqref{def.Galpha} of $G_\alpha$ at the last line. Finally using the hypothesis $\|g_{k(2\alpha-d)}*\varphi\|\le 1$ repeatedly for each edge of the tree $\pi(\mathfrak t)$, we deduce that 
$$\mathbb E\Big[\Big(\sum_{x\in \mathcal C_0^1 \cap \dots \cap \mathcal C_0^k} \varphi(x)\Big)^n\Big] 
\le (C^{2k+1})^n\, (n!)^k, $$
and the desired result follows. 
\end{proof}
We then conclude the proof of our main result here. 

\begin{proof}[Proof of Theorem~\ref{theo.generalcluster}]
Assume first that $\alpha \in(\tfrac{(k+1)d}{2k},\tfrac{kd}{2(k-1)})$, for some $k\ge 2$. 
In this case, we can follow a similar argument as in the proof of Theorem~\ref{LD.intersection.finite}, and we keep the same notation. Recall in particular that $R_t=\exp(t^{\frac{2\alpha}{d}-1})$, and $\mathcal C_t(0) = \mathcal C_0 \cap B(0,R_t)$.  
In order to apply  Theorem~\ref{theo.generalmoment}, 
we fix $\mathcal C_t(0)$ and consider $\varphi (x) = \1\{x\in \mathcal C_t(0)\}/\sup_{z\in V} g_{(k-1)(2\alpha-d)}(z,\mathcal C_t(0))$. One then has for any $t>0$,
$$\mathbb P(|\mathcal C_0^1 \cap \dots\cap C_0^{k-1} \cap \mathcal C_t(0)|>t)\le 2\, \mathbb E\left[\exp\Big(- \frac{\kappa\,  t^{\frac 1{k-1}}}{\sup_{z\in V} g_{(k-1)(2\alpha-d)}(z,\mathcal C_t(0))^{\frac 1{k-1}}}\Big)\right]. $$ 
Recall now the definition of the set $\mathcal E$ from~\eqref{def.E}. It can readily be seen,  following the same steps as in the proof of Theorem~\ref{LD.intersection.finite}, 
that on the event $\mathcal E$, when $\alpha\in (\frac{(k+1)d}{2k},\tfrac{kd}{2(k-1)})$, one has 
$$\sup_{z\in V} g_{(k-1)(2\alpha-d)}(z,\mathcal C_t(0)) \le C t^{1-(k-1)(\frac{2\alpha}{d}-1)},$$
and when $\alpha = \tfrac{kd}{2(k-1)}$, for some $k\ge 3$, then 
$$\sup_{z\in V} g_{(k-1)(2\alpha-d)}(z,\mathcal C_t(0)) \le C\log t,$$ 
which altogether conclude the proof of  Theorem~\ref{theo.generalcluster}.
\end{proof}


\section{Proofs of Theorems~\ref{theo.SRW},~\ref{thm.BRW} and~\ref{theo.SRWBRW}}\label{sec.BRW}
We start with the proof of Theorem~\ref{thm.BRW}. There are two parts, one about the upper bounds, which follows similar arguments as for Theorem~\ref{theo.generalcluster} (at least concerning the critical BRWs, some additional care is needed to handle the case of walks indexed by $\mathcal T_\infty$), and one about the lower bounds, which requires some new argument, based on the notion of waves introduced in~\cite{AS24+}. We then conclude this section with a few words explaining the minor changes needed for the proofs of Theorems~\ref{theo.SRW} and~\ref{theo.SRWBRW}. 

\subsection{Proof of Theorem~\ref{thm.BRW}: upper bounds} 
Let us start with the proof in dimension $d\ge 9$, for which it suffices to treat the case of two walks indexed by $\mathcal T_\infty$. All one needs here is an analogue of Theorem~\ref{thm.expmoment} for $\mathcal R_\infty$, which is precisely given in~\cite{ASS23}. Then the rest of the proof is exactly the same as for  Theorem~\ref{LD.intersection.finite}. 

\vspace{0.2cm}
Concerning the fact that the tail distribution for the intersection of critical BRWs is bounded by the tail distribution for the intersection of walks indexed by $\mathcal T_\infty$, this follows from the fact that $\mathcal R_c$ and $\mathcal R_\infty$ can be coupled in such a way that $\mathcal R_c\subset \mathcal R_\infty$. 
This is immediate when $\mathcal T_\infty$ is the invariant infinite tree, since in this case by definition a copy of $\mathcal T_c$ is attached to the root, and in case of Kesten's tree, one can notice that all vertices on the spine produce offspring according to the size biased distribution, which stochastically dominates $\mu$, while the offspring distribution of other vertices is $\mu$, therefore $\mathcal T_\infty$ stochastically dominates $\mathcal T_c$ in this case as well.  

\vspace{0.1cm}
We now move to the lower dimensions, where one needs to consider the intersection of more than two ranges. For this one needs the following result about local times of BRWs,  which might be of independent interest. Recall~\eqref{localtime.def} and the notation introduced in Section~\ref{sec.manyclusters}. We denote by $\mathbb P_z$ the law of a BRW starting from $z$, i.e.~taking value $z$ at the root, and let $\mathbb E_z$ denote the corresponding expectation. Given $\mathfrak t\in \mathbb T_n^\circ$, $\underline x=(x_1,\dots,x_n) \in (\mathbb Z^d)^n$, and $z\in \mathbb Z^d$, we define $\widetilde G(\mathfrak t,\underline x,z)$, as follows. If the root of $\mathfrak t$ has $L\ge 2$ children, to which emanate some trees $\mathfrak t_1,\dots,\mathfrak t_L$, with labels in $I_1,\dots,I_L$, then we let 
$$\widetilde G(\mathfrak t,\underline x,z) 
= f_{L+1}(x_{\mathfrak t_1}, \dots,x_{\mathfrak t_L},z)
\cdot \prod_{\ell = 1}^L G_{d-2}(\mathfrak t_\ell,(x_i)_{i\in I_\ell}),$$ 
where $f_{L+1}$ is the function introduced before Proposition~\ref{prop.Green}.

\begin{proposition}\label{prop.localtimesBRW}
There exists $C>0$, such that for any $n\ge 1$, and any $x_1,\dots,x_n,z \in \mathbb Z^d$ (possibly with repetition), with $d\ge 5$, 
\begin{equation}\label{LT.BRWc}
\mathbb E_z\Big[\prod_{i=1}^n \ell_c(x_i)\Big] \le C^n \sum_{\mathfrak t\in \mathbb T_n} G_{d-2}(\mathfrak t,\underline x)\cdot g_{d-2}(x_{\mathfrak t},z),
\end{equation}
and 
\begin{equation}\label{LT.BRWinfinity}
\mathbb E_z\Big[\prod_{i=1}^n \ell_\infty(x_i)\Big] \le C^n \Big(\sum_{\mathfrak t\in \mathbb T_n^\circ} \widetilde G(\mathfrak t,\underline x,z) + \sum_{\mathfrak t\in \mathbb T_n} G_{d-2}(\mathfrak t,\underline x)\cdot g_{d-4}(x_{\mathfrak t},z)\Big). 
\end{equation}
\end{proposition}
\begin{proof}
We first prove the result for critical BRWs. 
For this we use a similar induction argument as in the proof of Proposition~\ref{prop.npoint}, so let us just mention the minor changes needed here. 
The result for $n=1$ is immediate. Assuming it is true for $k\le n-1$, let now $x_1,\dots,x_n$ be given. Define $\text{MRCA}(v_1,\dots,v_n)$ as  the most recent common ancestor of $v_1,\dots,v_n\in \mathcal T_c$. We write 
\begin{align}\label{eq.MRCA}
\nonumber & \mathbb E_z\Big[\prod_{i=1}^n \ell_c(x_i)\Big]  =\mathbb E\Big[ 
 \sum_{v_1,\dots,v_n\in \mathcal T_c} \1\{S_{v_1}=x_1,\dots,S_{v_n}=x_n\}\Big] \\
\nonumber & = \mathbb E_z\Big[ 
 \sum_{v_1,\dots,v_n\in \mathcal T_c}
 \sum_{v_0\in \mathcal T_c} \1\{S_{v_1}=x_1,\dots,S_{v_n}=x_n, v_0=\text{MRCA}(v_1,\dots,v_n), v_0\in \{v_1,\dots,v_n\}\Big]\\
 & + \mathbb E_z\Big[ 
 \sum_{v_1,\dots,v_n\in \mathcal T_c}
 \sum_{v_0\in \mathcal T_c} \1\{S_{v_1}=x_1,\dots,S_{v_n}=x_n, v_0=\text{MRCA}(v_1,\dots,v_n), v_0\notin \{v_1,\dots,v_n\}\Big]. 
\end{align}
Using the induction hypothesis, one can bound the first term on the right-hand side of~\eqref{eq.MRCA} by 
\begin{align*}
C\sum_{i=1}^n g_{d-2}(x_i,z)\cdot \mathbb E_{x_i}\Big[ \prod_{j\neq i} \ell_c(x_j)\Big] &\le C^n\sum_{i=1}^n g(x_i,z) \sum_{\mathfrak t\in \mathbb T_{n-1}} G_{d-2}(\mathfrak t,(x_j)_{j\neq i})\cdot g_{d-2}(x_{\mathfrak t},x_i) \\
& \le C^n \sum_{\mathfrak t\in \mathbb T_n^\bullet} G_{d-2}(\mathfrak t,\underline x)\cdot g_{d-2}(x_{\mathfrak t},z). 
\end{align*}
Concerning the second term in~\eqref{eq.MRCA}, it is equal to  
\begin{align*}
\sum_{j\ge 2}\mu(j)\sum_{L=2}^{j\wedge n} \binom{j}{L}  \sum_{I_1,\dots,I_L} \sum_{y\in \mathbb Z^d }g(y-z)\prod_{\ell=1}^L \mathbb E_{\nu_y}\Big[\prod_{i\in I_\ell} \ell_c(x_i)\Big],
\end{align*}
where the third sum is over partitions $I_1,\dots,I_L$ of $\{1,\dots,n\}$, with $I_\ell$ nonempty for all $\ell$, and $\nu_y$ is the uniform measure on the neighbors of $y$. Using then that $\mu$ has a finite exponential moment by hypothesis, we get that 
$\sum_{j\ge L} \mu(j)\binom{j}{L} \le c^L$, for some $c>0$. Then exactly as in the proof of Proposition~\ref{prop.npoint} we deduce that the last sum is bounded by 
$$C^n \sum_{\mathfrak t\in \mathbb T_n^\circ} G_{d-2}(\mathfrak t,\underline x)\cdot g_{d-2}(x_{\mathfrak t},z),$$ 
which altogether proves~\eqref{LT.BRWc}.

Now following a similar argument as in Section 4.2 of~\cite{ASS23}, taking advantage of the fact that $\mathcal T_\infty$ is made of a spine to which are attached critical BGW trees (actually adjoint BGW trees, but this is a minor point), one can deduce from~\eqref{LT.BRWc} that for any $x_1,\dots,x_n,z\in \mathbb Z^d$, 
\begin{align}\label{Linfty}
\nonumber \mathbb E_z\Big[\prod_{i=1}^n \ell_\infty(x_i)\Big] & \le C^n 
\sum_{L=1}^n \sum_{(I_1,\dots,I_L)} \sum_{(\mathfrak t_1,\dots,\mathfrak t_L)} \sum_{y_1,\dots,y_L} g_{d-2}(z,y_1)\\
& \qquad \times \Big(\prod_{\ell = 1}^{L-1}g_{d-2}(y_\ell,y_{\ell +1})\Big)\cdot \Big(\prod_{\ell=1}^L g_{d-2}(x_{\mathfrak t_\ell},y_\ell)\Big) \cdot \Big(\prod_{\ell = 1}^L G_{d-2}(\mathfrak t_\ell,(x_i)_{i\in I_\ell})\Big),
\end{align} 
where the second sum is now over ordered partitions $(I_1,\dots,I_L)$ of $\{1,\dots,n\}$, with $I_\ell$ nonempty for all $\ell$, and the third sum is over ordered families of trees $(\mathfrak t_1\in \mathbb T_{|I_1|}, \dots, \mathfrak t_L\in \mathbb T_{|I_L|})$. Then the sum corresponding to $L=1$ is upper bounded by $C\sum_{\mathfrak t\in \mathbb T_n} g_{d-4}(x_{\mathfrak t},z) \cdot G_{d-2}(\mathfrak t,\underline x)$, while the sum over $L\ge 2$ is upper bounded by $C\sum_{\mathfrak t\in \mathbb T_n^\circ} \widetilde G(\mathfrak t,\underline x,z)$, thanks to  Proposition~\ref{prop.Green}, which concludes the proof of~\eqref{LT.BRWinfinity}. 
\end{proof}

Now that we have Proposition~\ref{prop.localtimesBRW} at hand, the proofs of the upper bounds in Theorem~\ref{thm.BRW} are exactly the same as for Theorem~\ref{theo.generalcluster}.

\subsection{Proof of Theorem~\ref{thm.BRW}: lower bounds} \label{subsec.lowerbound}
We now move to the lower bounds for the intersection of  independent critical BRWs.  
For this we need some new notation. Given two vertices $u,v\in \mathcal T_c$, we write $v\le u$ when $v$ is on the geodesic from $u$ to the root, denoted by $\varnothing$, and write $v<u$ when in addition $v$ is different from $u$. In this case, we set $[v,u) = \{w\in \mathcal T_c : v\le w< u\}$. Also for a set $\Lambda\subset \mathbb Z^d$, we let $\partial \Lambda$ the exterior boundary of $\Lambda$, i.e. the set of vertices not in $\Lambda$ which have at least one neighbor in $\Lambda$. The lower bounds are immediate consequences of the following proposition. 

\begin{proposition}\label{prop.lowerbound}
Let $\rho\in (0,1)$. There exists $c>0$, such that for any $r\ge 1$, and any subset $A\subset B(0,r)$, with $|A|\ge \rho |B(0,r)|$, one has 
$$\mathbb P(|\mathcal R_c\cap A|>  |A|/2) \ge \exp(-c\, r^{d-4}). $$  
\end{proposition}
\begin{proof}
Fix $\rho\in (0,1)$, and let $r\ge 1$ and $A\subset B(0,r)$, such that $|A|\ge \rho|B(0,r)|$. 
Define then inductively four sequences of vertices of $\mathcal T_c$, denoted $(\eta_n)_{n\ge 0}$, $(\overline{\eta}_n)_{n\ge 0}$, $(\widetilde \eta_n)_{n\ge 0}$ and $(\nu_n)_{n\ge 0}$ as follows. First, $\eta_0=\overline{\eta}_0=\{\varnothing\}$. 
Secondly, for any $n\ge 0$, $\overline \eta_n$ and $\widetilde \eta_n$ are subsets of $\eta_n$, such that $\widetilde \eta_n = \eta_n \setminus \overline \eta_n$. In short, $\widetilde \eta_n$ is a set of saved particles, which will be used for covering purposes, whereas $\overline \eta_n$ is used to create the next wave.  To be more concrete, let $\widetilde \eta_n$ a set of $\lfloor \min(r^2,|\eta_n|/2)\rfloor$ particles of $\eta_n$ chosen uniformly at random. 
Finally, for any $n\ge 0$,  
$$\nu_n = \{ u\in \mathcal T_c : S_u \in \partial B(0,2r), \textrm{ and }\exists v\in \overline{\eta}_n \textrm{ satisfying }v<u \textrm{ and } S_w\in B(0,2r) \ \forall w\in [v,u)\},$$
and
$$\eta_{n+1} =\{ u\in \mathcal T_c : S_u \in  B(0,r), \textrm{ and }\exists v\in \nu_n \textrm{ satisfying }v<u \textrm{ and } S_w\in \mathbb Z^d\setminus B(0,r) \ \forall w\in [v,u)\}.$$
Note that since $\mathcal T_c$ is finite almost surely, all these sets  are empty for $n$ large enough. 
For $n\ge 0$, we define $\mathcal T_n$ as the subtree of  $\mathcal T_c$, whose vertices are those for which there are no vertices in $\eta_n$ on the geodesic between them and the root, except possibly themselves. In other words, identifying $\mathcal T_n$ with its set of vertices by a slight abuse of notation, one has
$$\mathcal T_n = \{u\in \mathcal T_c : [\varnothing,u)\cap \eta_n = \emptyset\}. $$ 
We next define the filtration $(\mathcal F_n)_{n\ge 0}$, by letting for each $n\ge 0$, 
$$\mathcal F_n=\sigma\Big(\mathcal T_n, \widetilde \eta_1,\dots,\widetilde \eta_n, \{(u,S_u) : u\in \mathcal T_n\}\Big).$$
Also, for $u\in \mathcal T_c$, let $$\mathcal R(u)= \{S_v : u\le v \}.$$ 
Then for $n\ge 0$, let 
$$\mathcal R(\widetilde \eta_n) =  \bigcup_{u\in \widetilde \eta_n} \mathcal R(u),\qquad \widetilde{\mathcal R}_n = \bigcup_{k=0}^n \mathcal R(\widetilde \eta_k).$$
Define furthermore,  
$$\tau = \inf\big\{n \ge 0 : |\widetilde{\mathcal R}_n\cap A|> \frac{|A|}{2}\big\}.$$
Let now $\nu\in (0,1)$ and $\kappa>0$ be some constants to be fixed later, and consider for each $n\ge 1$, with $\Delta_n= \big|\widetilde{\mathcal R}_{n-1}^c \cap A \cap \mathcal R(\widetilde \eta_n) \big|$, 
$$A_n = \big\{|\widetilde \eta_n| \ge  \nu \, r^2\big\}, \qquad B_n = \big\{\Delta_n>\kappa\, r^4\big\}.$$ 
We note that for any $n\ge 1$, almost surely on the event $A_n\cap \{\tau>n\}$, one has with $g(r) = r^{2-d}$, and for some constant $c_1>0$, 
\begin{align*}
    \mathbb E\left[\Delta_n \mid \mathcal F_n\right] & = \sum_{x\in \widetilde{\mathcal R}_{n-1}^c \cap A} \mathbb P(x\in \mathcal R(\widetilde \eta_n) \mid \mathcal F_n) \\
    & \ge \frac{|A|}{2} \cdot \Big(1- (1-c_1g(r))^{\nu \, r^2}\Big), 
\end{align*}
using that uniformly over $x,z\in B(0,r)$, the probability that a BRW starting from $z$ hits $x$ is at least $c_1\, g(r)$, for some constant $c_1>0$ (this can be seen using an elementary second moment method, see also~\cite{Zhu} for a finer result). It follows that for some constant $c_2>0$, one has almost surely on $A_n\cap \{\tau>n\}$, 
\begin{equation}\label{lowerDelta1}
\mathbb E\left[\Delta_n \mid \mathcal F_n\right]\ge c_2 \nu\, r^4. 
\end{equation}
Therefore, taking $\kappa = \tfrac{c_2\nu}{2}$, we get that almost surely on $A_n\cap \{\tau>n\}$, 
\begin{equation}\label{lowerDelta2}
\mathbb P(B_n\mid \mathcal F_n) \ge \mathbb P\Big(\Delta_n \ge \frac 12 \mathbb E[\Delta_n\mid \mathcal F_n] \ \Big|\ \mathcal F_n\Big)\ge \frac 14 \cdot \frac{\mathbb E[\Delta_n\mid \mathcal F_n]^2}{\mathbb E[\Delta_n^2\mid \mathcal F_n] }, 
\end{equation}
using Paley-Zygmund's inequality at the end. 
Observe next that uniformly over $u\in \eta_n$, one has for some constant $C>0$, 
$$\mathbb E\big[|\mathcal R(u)\cap B(0,r)|\big]\le C \sum_{x\in B(0,2r)} g(\|x\|) \le C\, r^2,$$
and using the many-to-two formula,  
$$\mathbb E\big[|\mathcal R(u)\cap B(0,r)|^2\big]\le 
C \sum_{x\in \mathbb Z^d} \sum_{y,z\in B(0,2r)}  g(\|x\|)g(\|x-y\|)g(\|z-x\|) \le C\, r^6. $$  
It follows that for some constant $c_3>0$, almost surely 
\begin{equation}\label{upper1}
\mathbb E[\Delta_n^2 \mid \mathcal F_n] \le 
|\widetilde \eta_n|^2 \cdot \sup_{u\in \widetilde \eta_n} \mathbb E\big[|\mathcal R(u)\cap B(0,r)|\big]^2 + |\widetilde \eta_n| \cdot \sup_{u\in \widetilde \eta_n} \mathbb E\big[|\mathcal R(u)\cap B(0,r)|^2\big]\le c_3 \, r^8. 
\end{equation}
Combining~\eqref{lowerDelta1},~\eqref{lowerDelta2} and~\eqref{upper1} we get that almost surely on $A_n\cap \{\tau > n\}$, for some constant $c_4>0$ (depending on $\rho$),
$$\mathbb P(B_n \mid \mathcal F_n) \ge c_4. $$ 
On the other hand, if $\nu$ is sufficiently small, then Paley-Zygmund inequality again yields the existence of a constant $c_5>0$ (depending on $\nu$), such that for $n\ge 1$, almost surely on $A_n$, 
$$\mathbb P(A_{n+1}\mid \mathcal F_n)\ge c_5,$$
and also 
$$\mathbb P(A_1)\ge \frac{c_5}{r^2},$$
see e.g.~\cite[Lemma 10.1]{AS24+} for a more detailed proof. 
To conclude we note that for any $n\ge 1$, conditionally on $\mathcal F_n$, the events $A_{n+1}$ and $B_n$ are independent, and therefore, for any $N\ge 1$, by induction, 
$$\mathbb P\Big( \bigcap_{n=1}^{N\wedge \tau} A_n\cap B_n\Big)\ge   \frac{(c_4c_5)^N}{r^2}. $$ 
Let now $C>0$ be such that $\kappa Cr^d>|A|$. Taking then $N= \lfloor Cr^{d-4}\rfloor$,  we obtain that for some constant $c>0$, 
$$\mathbb P(|\mathcal R_c \cap A|\ge  |A|/2)\ge \exp(-c\, r^{d-4}),$$
as needed. 
\end{proof}
Applying now the result repeatedly, first with $A=B(0,r)$, and then with $A=\mathcal R_c^1\cap\dots\cap \mathcal R_c^j\cap B(0,r)$, for $j\le i-1$, we obtain that 
$$\mathbb P\big(|\mathcal R_c^1\cap \dots \cap \mathcal R_c^i\cap  B(0,r)|\ge \rho\,  r^d\big)\ge \exp(-c\, r^{d-4}), $$
for some (possibly different and depending on $i$) constants $\rho, c>0$. 
Finally taking $r= (t/\rho)^{1/d}$, concludes the proof of the lower bounds in Theorem~\ref{thm.BRW}. \hfill $\square$

\subsection{Proofs of Theorems~\ref{theo.SRW}  and~\ref{theo.SRWBRW}}\label{sec.SRW.BRW}
Concerning the lower bounds, one can use an analogue of Proposition~\ref{prop.lowerbound} for simple random walks. We claim that for any fixed $\rho\in (0,1)$, there exists $c>0$, such that for any $r\ge 1$, any set $A\subset B(0,r)$, with $|A|\ge \rho |B(0,r)|$, and any $t\in (r^2,|A|/2)$, one has 
\begin{equation}\label{lower.SRW}
\mathbb P(|R_\infty \cap A|>t)\ge \exp(-c\, t/r^2). 
\end{equation}
A proof of this statement can be done following exactly the same lines as for Proposition~\ref{prop.lowerbound}, using excursions of a SRW between $B(0,r)$ and $\partial B(0,2r)$, instead of waves. 

\vspace{0.2cm}
Applying the result repeatedly with $r=Ct^{1/d}$, with $C$ some large enough constant, we get that for any $k\ge 2$ and in any dimension $d\ge 5$, 
$$\mathbb P(|R_\infty^1\cap \dots \cap R_\infty^k \cap B(0,Ct^{1/d})|> t) \ge \exp(-c\, t^{1-\frac 2d}),$$
for some constant $c>0$, which already gives the lower bounds in Theorem~\ref{theo.SRW}. Applying additionally  Proposition~\ref{prop.lowerbound}, we get as well the lower bound in Theorem~\ref{theo.SRWBRW}~$(i)$. 

\vspace{0.2cm}
Now applying Proposition~\ref{prop.lowerbound} and~\eqref{lower.SRW} with $r=Ct^{\frac 1{d-2}}$, and $A=\mathcal R_c^1\cap \dots\cap \mathcal R_c^k\cap B(0,r)$, we obtain 
that for some constant $c>0$, and for any $t>1$, in dimension $d>5$,
\begin{equation*}
\mathbb P(|R_\infty\cap \mathcal R_c^1\cap \dots\cap \mathcal R_c^k|>t) \ge \exp(-c\, t^{1-\frac{2}{d-2}}),
\end{equation*}
proving the lower bound in Theorem~\ref{theo.SRWBRW}~$(ii)$. 

\vspace{0.2cm}
Let us consider the upper bounds now. Concerning Theorem~\ref{theo.SRW}, one can notice that by the Markov property, one has for any $d\ge 3$, and any $x_1,\dots,x_n\in \mathbb Z^d$,  
$$\mathbb P(x_1,\dots,x_n\in R_\infty) \le C^n \sum_{\sigma \in \mathfrak S_n} g_{d-2}(0,x_{\sigma(1)})\prod_{i=1}^{n-1} g_{d-2}(x_{\sigma(i)},x_{\sigma(i+1)}). $$ 
Since $|\mathfrak S_n|=n!$, we can then follow exactly the same argument as in the critical cases in the proof of Theorem~\ref{theo.generalcluster}. Next, concerning Theorem~\ref{theo.SRWBRW},  Part $(i)$ follows from the upper bound for two SRWs ranges, and for Part $(ii)$ it suffices to do it for the intersection of one SRW and one BRW ranges. For this one can use 
the same argument as for~\eqref{intersection.twodifferent}, see in particular Remark~\ref{rem.alphabeta}. We leave the details to the reader. 

\section{Intersection of  critical BRWs in low dimension: Proof of Theorem~\ref{BRW.lowdim}}\label{sec.lowdim}
In this Section, we provide upper and lower tail estimates for the distribution of the intersection of two BRW ranges in dimensions $d\le 8$. 
We start with a short heuristic discussion and then
divide this section into three parts. In the first part, we give lower bounds in dimensions $d\le 7$, where the tail distribution has polynomial decay.
In the second part, we establish the upper bounds, still in dimensions $d\le 7$, and finally 
in the third part, we discuss the critical case of dimension eight for which the decay is stretched exponential.

To understand the tail estimates, recall that the range 
$\mathcal R_c$ is a four dimensionsal random set,
so that conditioned on reaching the boundary of a ball $B(0,R)$, it typically fills a positive fraction of it 
in dimension $d\in \{1,2,3\}$,
fills a small fraction of order $1/\log(R)$ of it in dimension four, and covers a small density $\rho_d=R^{4-d}$, in higher dimension.
Thus, in dimension $d\in \{1,2,3\}$, it is enough that both BRWs reach the boundary of a ball of volume of order $t$ 
to produce the desired intersection; in dimension $d\in \{5,6,7\}$, the desired radius 
satisfies $\rho_d^2\cdot R^d=t$, which gives $R^{8-d}=t$; and in dimension four, we need $R^4$ of order $t\cdot \log^2(t)$. Finally, the probability of reaching the boundary of a ball $B(0,R)$ being of order $R^{-2}$ for each BRW, this yields the tail estimates from Theorem~\ref{BRW.lowdim}. 
Now, in dimension eight, which is critical for the intersection of two BRWs, the failure of $R^{8-d}=t$, makes the strategy for the two BRWs unclear. They might either  travel a very large distance (stretched exponential in $t$) to allow enough space for realizing the desired intersection, or produce a stretched exponential number of waves in a smaller ball, or use a mixture of these two strategies.

\subsection{Lower bounds in dimensions $d\le 7$} 
We use a conditional version of the second moment method. To be more precise, we use that for any nonnegative random variable $X$, and any event $A$, one has 
\begin{equation}\label{PZcond}
\mathbb P\Big(X\ge \frac 12 \mathbb E[X\mid A] \ \Big|\ A\Big) \ge \frac 14 \cdot \frac{\mathbb E[X\mid A]^2}{\mathbb E[X^2\mid A]}. 
\end{equation}
Assume first that $d\in \{5,6,7\}$, fix a constant $C>0$ to be specified later, let $r_t=Ct^{\frac 1{8-d}}$ and consider the event 
\begin{equation}\label{eventA}
A = \{\mathcal R_c\cap B(0,r_t)^c\neq \emptyset, \, \widetilde{\mathcal R}_c\cap B(0,r_t)^c\neq \emptyset\},
\end{equation}
that the two BRWs exit $B(0,r_t)$. Let also 
$$X=|\mathcal R_c\cap \widetilde{\mathcal R}_c \cap (B(0,2r_t)\setminus B(0,r_t))|.$$ 
It is well-known (see e.g.~\cite{AS24+}) that $\mathbb P(A) \asymp r_t^{-4}$. Hence, noting also that $X=0$ on $A^c$, one obtains that for some constant $c_1>0$, 
$$\mathbb E[X\mid A] = \frac{\mathbb E[X]}{\mathbb P(A)} \ge c_1 r_t^4 \sum_{x\in B(0,2r_t)\setminus B(0,r_t)} g_{d-2}(0,x)^2 \ge c_1\cdot r_t^{8-d}.$$ 
In particular one can fix $C$ large enough, so that $\mathbb E[X\mid A] \ge 2t$. We aim at applying~\eqref{PZcond} now, but for this it remains to bound the second moment of $X$. Given $x\in \mathbb Z^d$, we let $\ell_c(x) = \sum_{u\in \mathcal T_c} \1\{S_u = x\}$, the local time at $x$ for the critical BRW. Decomposing the event $\{S_u=x,S_v=y\}$ according to all possible positions of the walk at the youngest common ancestor of $u$ and $v$, say $z\in \mathbb Z^d$, 
we get for some constant $c_2>0$ (that might change from line to line),  
\begin{align}\label{2ndmoment}
\nonumber \mathbb E[X^2] & \le \sum_{x,y\in B(0,2r_t)\setminus B(0,r_t)} \mathbb P(x,y\in \mathcal R_c)^2 \le \sum_{x,y\in B(0,2r_t)\setminus B(0,r_t)} \mathbb E[\ell_c(x)\ell_c(y)] \\
\nonumber & \le c_2\sum_{x,y\in B(0,2r_t)\setminus B(0,r_t)}
\Big(\sum_{z\in \mathbb Z^d} g_{d-2}(0,z)g_{d-2}(z,x)g_{d-2}(z,y)\Big)^2\\
& \le c_2 r_t^{-2(d-2)} \sum_{x,y\in B(0,2r_t)} g_{d-4}(x,y)^2  \le c_2\cdot r_t^{12-2d}. 
\end{align}
Then~\eqref{PZcond} gives us that for some $c_3>0$,
$$\mathbb P(X>t) \ge \mathbb P(X>t\mid A)\cdot \mathbb P(A) \ge c_3\cdot  r_t^{-4},$$
which is the desired result.

The proof when $d\in \{1,2,3\}$ is similar. Let this time $r_t = Ct^{1/d}$, with $C>0$ to be fixed, and define $A$ and $X$ as before. We shall use now that $\mathbb P(x \in \mathcal R_c) \asymp \|x\|^{-2}$, see~\cite{LGL} for an even stronger result. Hence the same computation as before shows that 
$\mathbb E[X\mid A]\ge 2t$, if $C$ is chosen large enough. Concerning
the second moment of $X$, simply notice that $X\le |B(0,2r_t)| \cdot \1\{A\}$, whence by an application of~\eqref{PZcond}, we get for some constant $c>0$,  
$$\mathbb P(X\ge t)\ge c\cdot \mathbb P(A) \ge c\cdot t^{-4/d},$$
as wanted.

Let us finally consider the case of dimension four. Let $r_t= Ct^{1/4} \sqrt{\log t}$, with $C>0$ some constant to be fixed. 
Define also again $X$ and $A$ as before. We use that in dimension four, $\mathbb P(x\in \mathcal R_c)\asymp \frac 1{\|x\|^2\log \|x\|}$, as shown by Zhu~\cite{ZhuECP}, see also~\cite{Zhu3} for a finer result. This implies that for some constant $c>0$, 
$$\mathbb E[X\mid A] \ge \frac{c\, r_t^4}{(\log r_t)^2},$$
which is larger than $2t$ if $C$ is chosen large enough. Regarding now the second moment of $X$, one has for any $x,y\in B(0,2r_r)\setminus B(0,r_t)$, denoting by $u\wedge v$ the youngest common ancestor of vertices $u$ and $v$ in $\mathcal T_c$, and using the convention $\frac 10 = 1$, 
\begin{align*}
\mathbb P(x,y\in \mathcal R_c)& \le 
\sum_{z\in \mathbb Z^4} \mathbb P(\exists u,v\in \mathcal T_c : S_u = x, S_v =y, S_{u\wedge v}=z)\\
& \le C\sum_{z\in \mathbb Z^4} \frac{1}{\|z\|^2}\cdot \frac{1}{\|x-z\|^2\log \|x-z\|} \cdot \frac{1}{\|y-z\|^2 \log \|y-z\|}\\
&\le \frac{C}{r_t\|x-y\|\cdot (\log \|x-y\|)^2}, 
\end{align*}
with $C>0$ some constant. Taking the square, and summing over $x,y\in B(0,2r_t)\setminus B(0,r_t)$, we get 
$$\mathbb E[X^2] \le C\frac{r_t^4}{(\log r_t)^4}. $$
Altogether, this gives that for some constant $c>0$, 
$$\mathbb P(X>t) \ge c \cdot \mathbb P(A) \ge \frac{c}{r_t^4},$$
which is the desired result. 

\subsection{Upper bounds in dimensions $d\le 7$}
When $d\in \{1,2,3\}$, the upper bound follows from a simple application of Markov's inequality. Indeed, let $r_t = ct^{1/d}$, with $c$ small enough so that $|B(0,r_t)| \le t/2$. Then we just write, with the same notation as in~\eqref{eventA} for the event $A$, $$\mathbb P(|\mathcal R_c \cap \widetilde R_c|\ge t) \le \mathbb P(A) \le C r_t^{-4}, $$
for some constant $C>0$, which is the desired result.

For $d\in \{5,6,7\}$, set $r_t = t^{\frac 1{8-d}}$. One has for some constant $C>0$, using Markov's inequality, 
\begin{equation}\label{upper.lowdim1}
\mathbb P(|\mathcal R_c\cap \widetilde{\mathcal R}_c\cap B(0,r_t)^c|\ge t/2) \le \frac{2}{t} \sum_{x\in B(0,r_t)^c} \mathbb P(x\in \mathcal R_c)^2 \le \frac{C}{t} r_t^{4-d} \le Ct^{-\frac{4}{8-d}}. 
\end{equation}
On the other hand, if $d=5$, then the second moment computation from~\eqref{2ndmoment}, and Markov's inequality, give with $\mathcal S_i = B(0,2^{-i}r_t)\setminus B(0,2^{-i-1}r_t)$, 
\begin{align*}
\mathbb P(|\mathcal R_c\cap \widetilde{\mathcal R}_c\cap B(0,r_t)|\ge t/2)& \le
\sum_{i\ge 0}
\mathbb P(|\mathcal R_c\cap \widetilde{\mathcal R}_c\cap \mathcal S_i |\ge \sqrt{2}^{-(i+8)}t) \\
& \le \frac{C}{t^2}\sum_{i\ge 0} 2^i \cdot \mathbb E\big[|\mathcal R_c\cap \widetilde{\mathcal R}_c\cap \mathcal S_i |^2\big]\\
& \le \frac{C}{t^2} \sum_{i\ge 0} 2^{-i}r_t^2 \le \frac{C}{t^2} \cdot r_t^2 \le Ct^{-4/3},
\end{align*}
recalling that in dimension five $r_t = t^{1/3}$, for the last inequality.
Together with~\eqref{upper.lowdim1}, this concludes the proof of the upper bound in dimension five. 
If $d=6$, one needs a third moment bound (as the second moment of $X$ is not a growing function of $r_t$). For this one can use~\eqref{LT.BRWc}, and the facts that if $d=6$, 
$$\sup_{x\in \mathbb Z^6} \sum_{z\in B(0,2r)} g_{2(d-4)}(x,z)\le Cr^2, \quad \textrm{and}\quad  \sum_{z\in B(0,2r)\setminus B(0,r)} g_{2(d-2)}(0,z)\le Cr^{-2},$$ 
which altogether show that 
for some constant $C>0$, for any $r\ge 1$,
$$\mathbb E\big[|\mathcal R_c \cap \widetilde{\mathcal R}_c\cap (B(0,2r)\setminus B(0,r))|^3 \big]\le C r^2. $$
Hence, as before, since now $r_t = \sqrt{t}$, (using the same notation for $\mathcal S_i$ as above), 
\begin{align*}
\mathbb P(|\mathcal R_c\cap \widetilde{\mathcal R}_c\cap B(0,r_t)|\ge t/2)& \le
\sum_{i\ge 0}
\mathbb P(|\mathcal R_c\cap \widetilde{\mathcal R}_c\cap \mathcal S_i |\ge \sqrt{2}^{-(i+8)}t) \\
& \le 
\frac{C}{t^3} \cdot r_t^2 \le Ct^{-2}. 
\end{align*}
The argument in dimension seven is similar, except that one needs now to use a fifth moment (as lower moments of $X$ are not growing functions of $r_t$). Namely, we first note that when $d=7$, 
$$\sup_{x\in \mathbb Z^7} \sum_{z\in B(0,2r)} g_{2(d-4)}(x,z)\le Cr, \quad \textrm{and}\quad  \sum_{z\in B(0,2r)\setminus B(0,r)} g_{2(d-2)}(0,z)\le Cr^{-3}.$$ 
Then a similar argument as above gives for $r\ge 1$, 
$$\mathbb E\big[|\mathcal R_c \cap \widetilde{\mathcal R}_c\cap (B(0,2r)\setminus B(0,r))|^5 \big]\le C r, $$
whence we deduce, with $r_t= t$, 
$$\mathbb P(|\mathcal R_c\cap \widetilde{\mathcal R}_c\cap B(0,r_t)|\ge t/2) \le \frac{Cr_t}{t^5} \le Ct^{-4}. $$ 
Finally, if $d=4$, we write with $r_t = t^{1/4}\sqrt{\log t}$,  
$$\mathbb P(|\mathcal R_c\cap \widetilde{\mathcal R}_c\cap B(0,r_t)^c|\ge 1) \le \mathbb P(A) \le Cr_t^{-4},$$
and by the computation from the previous section, 
$$\mathbb P(|\mathcal R_c\cap \widetilde{\mathcal R}_c\cap B(0,r_t)|\ge t)\le \sum_{i\ge 0}
\mathbb P(|\mathcal R_c\cap \widetilde{\mathcal R}_c\cap \mathcal S_i |\ge 2^{-i-1}t) \le C\cdot \frac{r_t^4}{(\log r_t)^4}, 
$$
which altogether proves the desired upper bound.

\subsection{The dimension eight}\label{sec.dim8}
For the lower bound we apply Proposition~\ref{prop.lowerbound} twice with $r=Ct^{1/8}$, and $C$ some constant to be fixed. Indeed, applying it first with $A=B(0,r)$ for the first range $\mathcal R_c$, and then with  $A=B(0,r)\cap \mathcal R_c$, we deduce the desired lower bound provided $C$ is taken large enough. 

\vspace{0.2cm}
As for the upper bound, we present two proofs. The first one is based on our new bounds on the moments of local times. More precisely, 
let $R=\exp(t^{1/3})$. Using the Markov's inequality, we get
$$\mathbb P(| \mathcal R_c\cap \widetilde{\mathcal R}_c\cap B(0,R)^c|\ge 1) \le \mathbb E[| \mathcal R_c\cap \widetilde{\mathcal R}_c\cap B(0,R)^c|]= \sum_{x\in B(0,R)^c} \mathbb P(x\in \mathcal R_c)^2 \le CR^{-4}. $$  
Hence, it just remains to bound the probability that the intersection of the two clusters inside the ball $B(0,R)$ exceeds $t$.
Let $n\ge 1$, and note that Proposition~\ref{prop.localtimesBRW} and Cauchy--Schwarz inequality yield for some constant $C>0$, 
\begin{align*}
\mathbb E\Big[| \mathcal R_c\cap \widetilde{\mathcal R}_c\cap B(0,R)|^n \Big]
& =\sum_{x_1\dots,x_n\in B(0,R)}  \mathbb P(x_1,\dots,x_n\in \mathcal R_c)^2\\
&\le C^n n! \sum_{\mathfrak t\in \mathbb T_n} 
\sum_{x_1\dots,x_n\in B(0,R)} G_{2(d-2)}(\mathfrak t,\underline x),
\end{align*}
with the notation from Proposition~\ref{prop.localtimesBRW}. 
Moreover, for some (possibly larger) constant $C>0$,
$$\sup_{z\in \mathbb Z^8} \sum_{x\in B(0,R)}
 g_{2(d-4)}(z,x)\le C\log R. $$ 
As a consequence, one gets for any $n\ge 1$, and some constant $C>0$, 
$$ \mathbb E\Big[| \mathcal R_c\cap \widetilde{\mathcal R}_c\cap B(0,R)|^n \Big] \le C^n (n!)^2 (\log R)^n,$$
and thus for some positive constant $\kappa$, 
$$\mathbb E\Big[\exp\Big(\kappa  \sqrt{\frac{| \mathcal R_c\cap \widetilde{\mathcal R}_c\cap B(0,R)|}{\log R}}\Big) \Big] <\infty,$$ 
which concludes the proof, using Chebyshev's exponential inequality and the fact that $\sqrt{t/\log R}$ is of order $t^{1/3}$.

Our second proof follows the argument given in Section~\ref{subsec.proofmain}.  
For $i\ge 0$, let $\rho_i = 2^{-i}$, and let $r_i$ be such that 
$\rho_i r_i^4 = C_0 t^{1/3}$, with $C_0$ as in~\eqref{cond.rrho}. Then define the sets $(\Lambda_i)_{i\ge 0}$, as in~\eqref{setsLambdai}, and let $\mathcal E = \{|\Lambda_i|\le \rho_i^{-1}t^{2/3}, \textrm{ for all }i\ge 0\}$. 
Note that the number of relevant indices $i$ is of order $\log R$. Hence following the proof from Section~\ref{subsec.proofmain}, we deduce that on one hand $\mathbb P(\mathcal E^c) \le \exp(-\kappa t^{1/3})$, for some $\kappa>0$, and that on $\mathcal E$, for some $C>0$, 
$$\sup_{x\in \mathbb Z^8} g_{d-4}(x,\mathcal R_c\cap B(0,R))\le C (t^{2/3} + t^{1/3} \log R)\le Ct^{2/3},$$
implying the desired upper bound.

\section{Proof of Theorem~\ref{theo.scenario}}
\label{sec.scenario}
The proof mostly relies on the following result which provides a control on the size of the sets 
$\mathcal R_t(r,\rho)$, defined in~\eqref{Rtrrho}. It can be proved exactly as Corollary~\ref{cor.COR} or~\cite[Theorem 1.6]{AS21}, using the results  of~\cite{ASS23}, in particular Theorem 1.5 and its Corollary 1.6 there, so we omit the proof. 
\begin{theorem}
There exist positive constants $c$ and $C$, such that for any $r$, $\rho$ and $t$, satisfying
$$\rho r^{d-4} \ge C\cdot t^{\frac{d-4}{d-2}},$$
one has for any $L\ge 1$
$$\mathbb P(|\mathcal R_t(r,\rho)|\ge L)\le C\exp(-c\cdot \rho^{4/d}\cdot L^{1-\frac 4d}). $$
\end{theorem}
An immediate consequence of this result, together with the lower bound in Theorem~\ref{theo.SRWBRW} (ii), is that for any given $\rho>0$, one can find $\beta$ and $b$ large enough, so that
$$\lim_{t\to \infty} \mathbb P(|\mathcal R_{\beta t}(\beta t^{\frac 1{d-2}},\rho)|> b t^{\frac{d}{d-2}} \mid |R_\infty\cap \mathcal R_\infty|>t) = 0. $$ 
The second fact we should use is an analogue of Corollary~\ref{cor.1} for the SRW, which says that for any finite set $\Lambda\subset \mathbb Z^d$, with $d\ge 3$, one has 
$$\mathbb P(|R_\infty\cap \Lambda|>t)\le \exp\big(-\frac{c\cdot t}{\sup_{z\in \mathbb Z^d} g_{d-2}(z,\Lambda)}\big),$$
for some constant $c>0$, independent of $\Lambda$, see e.g. Lemma 2.1 in~\cite{AS23}. 
Note also that a simple first moment bound entails, for $\beta$ large enough, 
$$\lim_{t\to \infty} \mathbb P\big(R_\infty\cap \mathcal R_\infty \cap B\big(0,\exp((\beta t)^{\frac{d-4}{d-2}})\big)^c \neq \varnothing\mid |R_\infty\cap \mathcal R_\infty|>t\big)=0. $$ 
Then by decomposing the range $\mathcal R_\infty$ into sets $(\Lambda_i)_{i\ge 0}$, defined as in~\eqref{setsLambdai}, one may infer the second part of Theorem~\ref{theo.scenario}, by showing that for any $\varepsilon>0$, one may find $I>0$, such that asking $R_\infty$ to intersect $(\cup_{i\le I} \Lambda_i)^c$ in more than $\varepsilon t$ points is too costly; see also~\cite{AS23} for a similar argument. Finally, note that by definition, as soon as $\mathcal R_{\beta t}(\beta t^{\frac 1{d-2}},\rho)$ is nonempty, then the volume of $\mathcal R_{2\beta t}(2\beta t^{\frac 1{d-2}},\rho/2^d)$, must be of order at least $t^{d/(d-2)}$, which induces the lower bound in the first part of the theorem.

\section*{Appendix}
In this section we provide proofs of some of the results presented in the introduction and in Section~\ref{sec.cap}  about $\beta$-capacities. 
We start with the proof of a result quoted in the introduction. 

\begin{proof}[Proof of~\eqref{equiv.cap}]
A proof of this equivalence can for instance be found in the unpublished lecture notes~\cite{Khosh}. For the reader's convenience, we reproduce the short argument here. 
Let $\nu_0$ be a probability measure realizing the infimum in the definition of ${\rm Cap}_\beta(A)$. Fix $\eta\in (0,1)$, and let 
$$A_\eta = \big\{x\in A : g_\beta *\nu_0(x) < \frac{1-\eta}{{\rm Cap}_\beta(A)} \big\}. $$
Suppose that $\nu_0(A_\eta)>0$. Then consider the probability measure on $A_\eta$:  
$$\nu_\eta(\cdot) = \frac{\nu_0(\cdot \cap A_\eta)}{\nu_0(A_\eta)}. $$ 
For $\varepsilon\in (0,1)$, define
$$\mu_\varepsilon = (1-\varepsilon)\nu_0 + \varepsilon \nu_\eta. $$ 
 Define also the bilinear map on $\mathbb R^A$: 
$$(\varphi,\psi)\mapsto \langle \varphi,\psi\rangle_\beta = \sum_{x,y\in A} g_\beta(x,y)\varphi(x) \psi(y). $$ 
Writing $\mu_\varepsilon = \nu_0 - \varepsilon(\nu_0 - \nu_\eta)$, one has by bilinearity, 
$$\langle \mu_\varepsilon,\mu_\varepsilon\rangle_\beta
= \langle \nu_0,\nu_0\rangle_\beta - 2\varepsilon  \langle \nu_0,\nu_0-\nu_\eta \rangle_\beta + \varepsilon^2 \langle \nu_0-\nu_\eta,\nu_0-\nu_\eta\rangle_\beta. $$ 
Since by definition of $\nu_0$, one has 
$\langle \mu_\varepsilon,\mu_\varepsilon\rangle_\beta \ge 
 \langle \nu_0,\nu_0\rangle_\beta$, we deduce  
 $ 2\langle \nu_0,\nu_0-\nu_\eta \rangle_\beta \le \varepsilon   \langle\nu_0-\nu_\eta,\nu_0-\nu_\eta\rangle_\beta$. Letting $\varepsilon$ go to $0$, we get that 
 $$\langle\nu_0,\nu_0\rangle_\beta\le \langle\nu_\eta,\nu_0\rangle_\beta.$$
But by definition of $A_\eta$, the right-hand side is no more than $(1-\eta)\langle\nu_0,\nu_0\rangle_\beta$, yielding a contradiction. In conclusion, $\nu_0(A_\eta) = 0$, for any $\eta>0$. It then not difficult to show that $g_\beta *\nu_0(x) = \tfrac 1{{\rm Cap}_\beta(A)}$, for all $x\in A$. Consequently, if $\varphi(x) = {\rm Cap}_\beta(A)\cdot  \nu_0(x)$, one has $\|g_\beta * \varphi\|_\infty \le C$, for some constant $C>0$, independent of $A$, from which we infer 
$$\widetilde {\rm Cap}_\beta(A) \ge \frac 1C \cdot {\rm Cap}(A).$$
The inequality in the other direction is easier. Let $\varphi$ be a function realizing the maximum in the definition of $\widetilde {\rm Cap}_\beta(A)$. Using that $\|g_\beta *\varphi\|\le 1$, we deduce that $\langle \varphi,\varphi\rangle_\beta  \le \widetilde {\rm Cap}_\beta(A)$, and thus for all finite $A$, 
$$\widetilde {\rm Cap}_\beta(A) \le  {\rm Cap}_\beta(A),$$
concluding the proof of~\eqref{equiv.cap}. 
\end{proof}

\begin{proof}[Proof of Lemma~\ref{lem.cap.extract}]
The proof is similar to the proof of Theorem 1.1 in~\cite{AS23b}. The idea is to show, according to the probabilistic method, that a certain random set $\mathcal U$ satisfies the desired properties with positive probability. 
Let $\varphi$ be a maximizing function in the definition of 
$\widetilde{\rm Cap}_\beta\big(\cup_{x\in A} B(x,r)\big)$. For $x\in A$, define $\overline \varphi_x$, by  $$\overline \varphi_x = \frac{c}{r^\beta}\sum_{z\in B(x,r)} \varphi(z),$$
with $c>0$, chosen such that $\overline \varphi_x\le 1$, for all $x\in A$. This is possible since by definition, for each $x\in A$, one has by Lemma~\ref{cap.ball} and~\eqref{equiv.cap}, 
$$\sum_{y\in B(x,r)} \varphi(y) \le \widetilde{\rm Cap}_\beta(B(x,r)) \le Cr^{\beta}.$$
Consider now a sequence $(\xi_x)_{x\in A}$ of independent Bernoulli random variables with respective parameters $(\overline \varphi_x)_{x\in A}$, and define $\mathcal U=\{x\in A: \xi_x=1\}$. 
One has 
$$\mathbb E[| B(\mathcal U,r)|] \ge c' \cdot r^{d-\beta} \cdot {\rm Cap}_\beta(B(A,r)). $$
and 
$$\mathbb Var\big(| B(\mathcal U,r)|\big)\le \mathbb E\big[| B(\mathcal U,r)|\big].$$ 
Hence by Chebyshev's inequality, one has for $r$ large enough, 
\begin{equation}\label{card.U}
\mathbb P\Big(| B(\mathcal U,r)|\ge \frac{c'}{2} \cdot r^{d-\beta} \cdot {\rm Cap}_\beta( B(A,r))\Big) \ge \frac 34.
\end{equation}
It remains to show that with sufficiently high probability 
$${\rm Cap}_{\beta}\big( B(\mathcal U,r)\big)\ge c \cdot r^{\beta-d}\cdot | B(\mathcal U,r) |.$$
To see this, consider $\nu$ the uniform probability measure on $B(\mathcal U,r)$. By definition of $\beta$-capacity, one has 
\begin{equation}\label{lower.capU}
{\rm Cap}_{\beta}\big( B(\mathcal U,r)\big)
\ge \frac{|B(\mathcal U,r)|^2}{\sum_{y,z\in \mathcal B(\mathcal U,r)} g_\beta(y,z)}. 
\end{equation}
Now, taking expectation of the denominator gives 
\begin{align*}
& \mathbb E\Big[\sum_{y,z\in B(\mathcal U,r)}g_\beta(y,z)\Big] 
= \sum_{x,x'\in A} \overline \varphi_x\overline \varphi_{x'} \sum_{y\in B(x,r)} \sum_{z\in B(x',r)} g_\beta(y,z)\\
 & \le C\cdot\sum_{x\in A} r^d \cdot \overline \varphi_x\cdot  \Big(r^{d-\beta}\cdot  g_\beta*\varphi(x) +  r^{d-\beta}\cdot \overline{\varphi}_x \Big)\le C\cdot r^{2(d-\beta)} \cdot \sum_{z\in B(A,r)} \varphi(z) \\
 & \le C\cdot r^{2(d-\beta)} \cdot {\rm Cap}_{\beta}\big( B(A,r)\big). 
\end{align*} 
Therefore, Markov's inequality gives us that 
$$\mathbb P\Big(\sum_{y,z\in B(\mathcal U,r)}g_\beta(y,z)\le 4Cr^{2(d-\beta)}\cdot {\rm Cap}_{\beta}\big( B(A,r)\big) \Big)\ge \frac{3}{4}. $$
Together with~\eqref{card.U} and~\eqref{lower.capU}, we get that 
$$\mathbb P\Big({\rm Cap}_{\beta}\big( B(\mathcal U,r)\big)\ge \frac{(c')^2}{16 C}\cdot {\rm Cap}_{\beta}( B(A,r))\Big) \ge \frac 12,$$
and combining this with~\eqref{card.U} again concludes the proof of the corollary. 
\end{proof}

\end{document}